\documentclass[preprint,3p]{elsarticle}
\usepackage[toc,page,title,titletoc,header]{appendix}
\usepackage{stmaryrd}
\SetSymbolFont{stmry}{bold}{U}{stmry}{m}{n}
\usepackage{amsmath}
\usepackage{amssymb}
\usepackage{amsthm}
\usepackage{tabularx}
\usepackage{subfigure}
\usepackage{epsfig}
\usepackage{indentfirst}
\usepackage{cases}
\usepackage{graphicx}
\usepackage{color}
\usepackage{float}
\usepackage{multirow}
\usepackage{algorithmicx,algorithm}
\usepackage{gensymb}
\usepackage{caption}

\newcommand{\eps}{\varepsilon}
\newcommand{\dx}{\mathrm{d}x}
\newcommand{\ds}{\mathrm{d}s}

\newtheorem{theorem}{{Theorem}}
\newtheorem{lemma}[theorem]{{Lemma}}

\newcommand\vct[1]{\boldsymbol{#1}}
\newcommand{\rmd}{\mathrm{d}}

\title{An efficient diffusion generated motion method for wetting dynamics}
\author{
Song Lu \footnote{ lusong@lsec.cc.ac.cn.}
, Xianmin Xu\footnote{Corresponding author.  xmxu@lsec.cc.ac.cn.}
 \\
LSEC, ICMSEC, NCMIS,  Academy of Mathematics and Systems Science, Chinese Academy of Sciences, Beijing 100190, China; \\
 School of Mathematical Sciences, University of  Chinese Academy of Sciences, Beijing 100049, China\\
}
\date{}

\begin{document}

\begin{abstract}
By using the Onsager variational principle as an approximation tool, we develop
a new diffusion generated motion method for wetting problems. The method uses a signed distance function to
represent the interface between the liquid and vapor surface.
In each iteration, a linear diffusion equation with a linear boundary
condition is solved for one time step in addition to a simple re-distance step and a volume correction step.
The method has a {first-order} convergence rate with respect to the time step size even in the vicinity three-phase contact points.
Its energy stability property is analysed by careful studies for some geometric flows on substrates.
Numerical examples show that the method can be used to simulate  complicated wetting problems on inhomogeneous surfaces.
\end{abstract}
\maketitle


\section{Introduction}
Wetting describes how liquid stays and spreads on solid surfaces.
It is a fundamental two-phase flow problem and has important  applications in  many industrial processes, e.g. in coating,
printing, oil industry, etc. When the solid surface is homogeneous,
the contact angle between the liquid surface and the substrate is characterized
by {Young's equation\cite{Young1805}}. However, when the solid surface is
chemically inhomogeneous or geometrically rough, the wetting phenomena become
much more complicated. The apparent contact angle may be very different from {Young's} angle({so-called} the lotus effect). There also exist many meta-stable states
in the solid-liquid-vapor system{, and} this  generates  the interesting
contact angle hysteresis phenomenon.
Many theoretical and experimental studies have been done for the wetting problems(c.f. \cite{Gennes85,Bonn09,Gennes03}
and many {references} therein).

Mathematically, wetting is a free interface problem with complicated boundary conditions. A liquid drop can change shape or even topology on a solid substrate.
On the contact line, which is the line between the liquid-vapor interface and the solid surface, {Young's} equation
holds locally.
 Numerical simulation for the wetting problem is very challenging, especially for that on  rough or
 chemically inhomogeneous surfaces.
The standard methods based on a sharp-interface representation of the liquid-vapor interface are difficult
to deal with the topology change and the complicated boundary conditions  \cite{Ren2014,deckelnick2005computation,Brakke1992,womble1989front,zhao1996variational}.
On the other hand, the diffuse-interface model for the wetting problem includes
 a nonlinear relaxed boundary condition\cite{Qian:2003iz,xuwang2011}, which is also difficult
to deal with numerically. Moreover, in comparison with the phase-field model in other applications,
one needs  a very small interface thickness parameter in wetting problems,
since the parameter must be much smaller than the characteristic size of the roughness or chemical inhomogeneity
in order to simulate the wetting phenomena correctly.

Recently,  a threshold dynamics method has been developed  for wetting problems on
rough surfaces\cite{xuWangWang2016,xuWangWang2018}.
The method alternately diffuses and sharpens a linear combination of the characteristic functions of the liquid,
the vapor and the solid domains to decrease the total interfacial energy of the {three-phase} system. The method is  efficient
 since in each iteration one needs only to solve one or two linear diffusion equations
 in addition to a simple thresholding step and a volume correction step.
The diffusion equation can be solved by standard FFT method
 or {non-uniform} FFT techniques\cite{Jiang2018a,wang2019efficient}.
 However, the accuracy of the method is not very satisfactory, since
 one can only obtain a convergence rate of half order with respect to the time step $\delta t$(i.e. of order $O(\delta t^{1/2})$)
 near the three-phase contact points.  This is crucial in wetting problems since it is
 the contact angle condition at these points that determines the main physical properties of wetting.
To compute the wetting problem correctly,
one usually needs to choose {a} very small time step size and  very fine meshes,
especially for rough or chemically inhomogeneous surfaces.

The  threshold dynamic method for wetting in \cite{xuWangWang2016,xuWangWang2018} is basically {an} MBO type method\cite{merriman1992diffusion,ruuth1998efficient,ruuth1998diffusion}
and can be seen as a generalization of the method recently developed in \cite{esedoglu2015threshold}.
This  method {has} been studied a lot and also been used in many different problems \cite{ruuth2001convolution,esedog2006threshold,wang2016efficient,elsey2016threshold,osting2017generalized}. One can also {employ} {signed
distance functions} to replace the characteristic functions in some applications\cite{esedog2010diffusion,kublik2011algorithms}.
Main advantages of these methods are that they are easy to implement and  can
deal with the topological changes naturally.
A disadvantage is that the accuracy of the method is not very good for multi-phase free interface problems with
a triple junction.
In general, the convergence rate  with respect to the time step of the method is of order
$O(\delta t)$ for smooth curves while it is of order $O(\delta t^{1/2})$ when there is a triple junction\cite{zaitzeff2019voronoi}.
Some {second-order} threshold dynamics schemes have been developed for smooth curves in \cite{ruuth1998efficient,zaitzeff2020second}.
However, it is not clear if the methods work for multi-phase problems with  triple junctions.

In this paper, we develop a new diffusion generalized method using the signed distance function
for wetting problems. The method is efficient and easy to implement, since like in the standard method,
in each {iteration,} the main step is to solve a linear diffusion equation
with a linear boundary condition. {Meanwhile,} the method
achieves a {first-order} convergence rate $O(\delta t)$  with respect to the time step
even near three-phase contact points.
The main difference of the method from the previous ones in \cite{xuWangWang2016,xuWangWang2018} is that we do not include the
 solid phase domain in the diffusion equation. Therefore the three-phase  contact points
are not inside the computational domain but on the boundary
and the corresponding contact angle conditions
can be approximated with higher accuracy.

Another novelty of our work is  that we use the Onsager principle\cite{Onsager1931a,Onsager1931,DoiSoft} as an approximation tool
to derive a linear diffusion equation with a linear boundary condition for the liquid-vapor interface,
which is the main equation to be solved in the proposed numerical method.
The Onsager Principle is a fundamental variational principle in statistic physics\cite{Onsager1931,Onsager1931a,DoiSoft}.
Recently, it is found that the principle can be used as an powerful approximation tool in many problems in soft matter \cite{Doi2015,XuXianmin2016,DiYana2016,ManXingkun2016,ZhouJiajia2017,DiYana2018,guo2019onset,Jiang19b}.
We use the idea in a novel way to derive the linear equation for wetting problems.
This is done by assuming a {\it tanh} profile of the leading order approximation of a phase-field model for wetting problems and
choosing the signed distance function of the liquid-vapor interface as an unknown function. By applying the Onsager principle, we show that the signed distance function approximately satisfies a linear diffusion equation with a simple linear  boundary condition.



We  also give a stability analysis for the proposed method.
The analysis is based on an approximation of the method by some geometric flows and
 careful studies of their geometric properties. We
show that the total wetting energies {decay} when the time step is small.
Numerical examples show that the method works well for wetting problems on chemically inhomogeneous surfaces.
The contact angle hysteresis phenomena can be computed correctly. 

The rest part of the paper is organized as follows. In section 2, we  {give the Onsager principle briefly}
and the main idea to use it as an approximation tool. In section 3,
we introduce the mathematical models for wetting problems. In particular, we show that the Onsager
principle can be used to derive a modified Allen-Cahn equation with a relaxed boundary condition.
In section 4, we derive a linear diffusion equation for the signed distance function of
the liquid-vapor interface by using the Onsager principle to approximate
the modified Allen-Cahn equation. The diffusion equation is then used to construct
 a diffusion generated motion method for wetting problems.
 In section 5, we give some analysis {of} the energy stability property of our method.
In section 6, some numerical examples are given to show the method  has a {first-order} convergence rate
and works well for wetting problems on chemically inhomogeneous surfaces.
Finally, some {concluding} remarks are given in the last section.

%


\section{The Onsager Principle as an approximation tool}
The Onsager Principle is a variational principle in statistic physics\cite{Onsager1931,Onsager1931a}. It has been used to derive models in many problems in
soft matter science, such as the Stokes equation in hydrodynamics, the generalized Navier slip boundary condition
in moving contact line problems, the Nernst-Planck equation in electro-kinetics,
the Ericksen-Leslie equation in nematic liquid crystals, etc \cite{DoiSoft,Doi2011, QianTiezheng2006}. We will introduce
it briefly in this section and more details are referred to \cite{DoiSoft}.

Suppose a non-equilibrium physical system is characterized by a set of parameters $\alpha_i$, $i=1,\cdots, M$,
which may depend both on space and time.
When the inertial effect is ignored, the dynamics of the parameters can be determined
by using the Onsager principle as follows.
Firstly, define the  Rayleignian function  in the  system with respect to $\dot{\alpha}=\{\dot{\alpha}_1,\cdots,\dot \alpha_M\}$, which is the changing rate of the parameter $\alpha$, as
\begin{equation}\label{e:Rayleig}
\mathcal{R}(\alpha;\dot{\alpha})=\Phi(\dot{\alpha})+\dot{\mathcal E}(\alpha;\dot{\alpha}),
\end{equation}
where $\Phi(\dot{\alpha})$ is the energy dissipation function, which is the half of the
total energy dissipation rate in the system, and $\dot{\mathcal E}(\alpha,\dot{\alpha})$ is the changing rate of total energy $\mathcal{E}(\alpha)$.
{In general,} the energy dissipation function can be written as a quadratic form of $\dot{\alpha}$, i.e.,
$$\Phi(\dot{\alpha})=\frac{1}{2}\sum_{i,j=1}^M\zeta_{ij}\dot{\alpha}_{i}\dot{\alpha}_j,$$
where $\zeta_{ij}$ is a symmetric positive matrix. Here we ignore the spacial integration for simplicity.
Then the dynamics of $\alpha_i$, $i=1,\cdots, M$
can be determined by minimizing the total Rayleignian with respect to $\dot{\alpha}$, namely
\begin{equation}\label{e:Onsager}
\min_{\dot{\alpha}} \mathcal{R}(\alpha;\dot{\alpha}).
\end{equation}
This leads to a dynamic equation for $\alpha$,
\begin{equation}\label{e:dyn1}
\sum_{j=1}^M\zeta_{ij}\dot{\alpha}_j=-\frac{\delta \mathcal{E}}{\delta \alpha_i}.
\end{equation}
%
The equation is actually the  Euler-Lagrange equation  for \eqref{e:Onsager}.
In {the} next section, we will use the principle to derive a modified Allen-Cahn equation for {the} wetting problem.

Recently, it is found that the principle can be used as {a powerful} approximation tool\cite{Doi2015,XuXianmin2016,DiYana2016,ManXingkun2016}.
The key idea is to use Onsager principle only for a few key parameters(not the whole set $\alpha$) which {describe} the system approximately.
If the selected parameters depend only on time,
then we derive an ordinary differential equation for them.
In general, the ODE equation is much easier to solve than \eqref{e:dyn1} but still captures the main properties of the
complicated physical system. This method is very useful to approximate many free boundary problems
in two-phase flows, visco-elastic fluids, etc \cite{XuXianmin2016,DiYana2016,DiYana2018,guo2019onset}.
In section 4, we will use the idea to derive an efficient numerical method for wetting problems.

\section{Mathematical models of the wetting problem}
\label{sec:model}
\begin{figure}[!h]
  \centering
  \includegraphics[width=6cm]{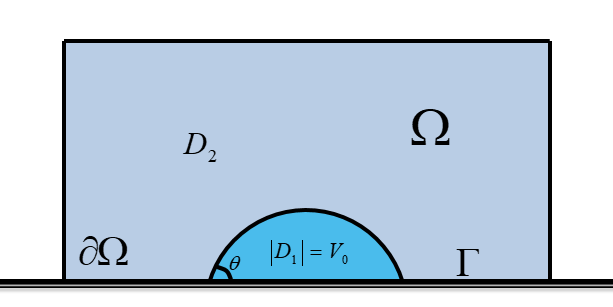}\qquad
  \includegraphics[width=6cm]{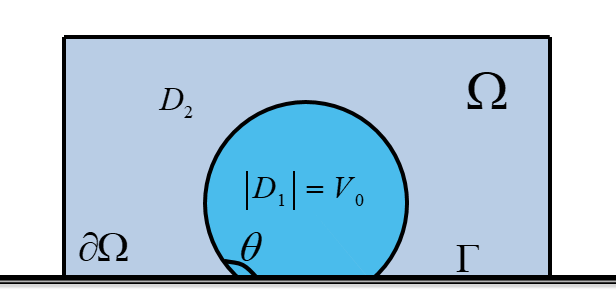}
  \caption{Wetting phenomena: a droplet on a hydrophilic substrate(left) and that on a  hydrophobic one(right)}\label{wetting}
\end{figure}
\subsection{The sharp-interface model}
   In wetting problems, if one ignores the gravity,
  the total energy of a liquid-vapor-solid system is composed of three interface energies, i.e.
 \begin{equation}\label{e:totalEng}
 \mathcal{E}=\gamma_{LV}|\Sigma_{LV}| +\gamma_{SL}|\Sigma_{SL}|+\gamma_{SV}|\Sigma_{SV}|.
 \end{equation}
 where $\gamma_{LV}$, $\gamma_{SL}$ and $\gamma_{SV}$ are respectively
 the energy densities of the liquid-vapor interface $\Sigma_{LV}$, the solid-liquid interface $\Sigma_{SL}$ and the solid-vapor
 interface $\Sigma_{SV}$.  The stationary profile of a liquid droplet on a solid surface is determined
 by minimizing the total energy under the constraint that the volume of the droplet is fixed.

To specify the mathematical model, we  suppose the liquid domain is $D_1\subset\Omega$ and the vapor domain is $D_2=\Omega\setminus D_1$, as shown in Figure~\ref{wetting}. The solid surface  $\Gamma$  is the lower boundary of $\Omega$.
The liquid-vapor interface is  given by $\Sigma_{LV}=\partial D_1\cap\partial D_2$,
 the solid-liquid interface is  $\Sigma_{SL}=\partial D_1\cap \Gamma$ and the solid-vapor
 interface is  $\Sigma_{SV}=\partial D_2\cap\Gamma$.  Then the mathematical problem of wetting is
 \begin{equation}\label{e:sharpPb0}
\min_{D_1\cup D_2=\Omega, |D_1|=V_0}  \mathcal{E}(D_1,D_2):=\gamma_{LV}|\partial D_1\cap\partial D_2| +\int_{\partial D_1\cap \Gamma} \gamma_{SL}(s) \ds +\int_{\partial D_2\cap\Gamma}\gamma_{SV}(s) \ds.
 \end{equation}
Here we write the surface energy on the solid boundary $\Gamma$ into an integral since the energy densities $\gamma_{SL}$ and $\gamma_{SV}$
might not be constant when the solid surface is chemically inhomogeneous.

When the solid surface is planar and homogeneous, the energy minimizing problem \eqref{e:sharpPb0} has
a unique minimizer corresponding to a spherical droplet with a contact angle, which is the angle between the liquid-vapor interface
and the solid surface, given by {Young's} equation\cite{Young1805}:
\begin{equation}\label{e:Young}
\gamma_{LV}\cos\theta_Y=\gamma_{SV}-\gamma_{SL}.
\end{equation}
When the solid substrate is geometrically rough or chemically inhomogeneous, there are many local minimizers
for the problem \eqref{e:sharpPb0}. The corresponding Euler-Lagrange equation
is a free boundary problem with multi-scale boundary conditions, in which the microscopic
inhomogeneity  affects the apparent contact angle dramatically.
Both  analysis and numerical simulations are difficult for the wetting problem on a rough or chemically inhomogeneous surface.

The problem~\eqref{e:sharpPb0} can be rewritten into an equivalent form.  Denote $C_0=\frac{1}{2}\int_{\Gamma}{\gamma_{SL}(s)+\gamma_{SV}(s)}\ds$ as a constant depending only on the property of the solid boundary $\Gamma$.  The total wetting energy is rewritten
as
 \begin{align}
 \mathcal{E}(D_1,D_2)&=\gamma_{LV}|\partial D_1\cap\partial D_2| +\int_{\partial D_1\cap \Gamma} \frac{\gamma_{SL}(s)-\gamma_{SV}(s)}{2} \ds +\int_{\partial D_2\cap\Gamma}\frac{\gamma_{SV}(s)-\gamma_{SL}(s)}{2} \ds+C_0\nonumber\\
 &=\gamma_{LV}|\partial D_1\cap\partial D_2| -\frac{\gamma_{LV}}{2} \int_{\partial D_1\cap \Gamma} \cos\theta_Y(s)\ds +\frac{\gamma_{LV}}{2}\int_{\partial D_2\cap\Gamma}\cos\theta_Y(s) \ds+C_0.\label{e:sharpEng1}
 \end{align}
 Here we have used {Young's} equation~\eqref{e:Young} and regard $\theta_Y(s)$ as a property of the substrate.
 Then the problem~\eqref{e:sharpPb0} is equivalent to the following one
\begin{equation}\label{e:sharpPb1}
\min_{D_1\cup D_2=\Omega, |D_1|=V_0}  \tilde{\mathcal{E}}(D_1,D_2):=\sigma|\partial D_1\cap\partial D_2| -\frac{\sigma}{2} \int_{\partial D_1\cap \Gamma} \cos\theta_Y(s)\ds +\frac{\sigma}{2}\int_{\partial D_2\cap\Gamma}\cos\theta_Y(s) \ds,
\end{equation}
where $\sigma$ is a given positive parameter, which can be seen as a dimensionless surface tension.

\subsection{The phase-field model}
Since the sharp-interface model described above is difficult in analysis and numerical simulations, it
is convenient to consider a phase-field approximation for the problem~\eqref{e:sharpPb1}.
The idea is to use a smooth phase field function $\phi_{\eps}$ to approximate the
liquid-vapor interface $\partial D_1\cap\partial D_2$ \cite{modica1977limite}. The total interface energy in \eqref{e:sharpPb1} can be approximated by
\begin{equation}\label{e:diffuEng}
\mathcal{E}_{\eps}(\phi_\eps)=\int_{\Omega}\frac{\eps}{2}|\nabla \phi_\eps|^2 +\frac{1}{\eps}f(\phi_{\eps}) \dx+\int_{\Gamma}
g(\phi_\eps) \ds,
\end{equation}
where $f(\phi_{\eps})=\frac{(1-\phi_{\eps}^2)^2}{4}$ and $g(\phi_\eps)=-\frac{\sigma}{4}\cos\theta_Y(s)(3\phi_{\eps}-\phi_{\eps}^3)$
with $\sigma=\frac{2\sqrt{2}}{3}$. The phase-field model for the stationary wetting problem is
\begin{equation}\label{e:phasefieldPb0}
\inf_{\int_{\Omega}\phi_\eps dx=C_1} \mathcal{E}_{\eps}(\phi_\eps).
\end{equation}
Here $C_1$ is a given constant to make sure the volume of the liquid domain $|\{x|\phi_\eps(x)>0\}|=V_0$.
It has been shown that the energy functional $\mathcal{E}_{\eps}$ $\Gamma$-converges to
the sharp-interface energy $\tilde{\mathcal{E}}$ defined in \eqref{e:sharpPb1} when $\eps$ goes to zero \cite{modica1987gradient,xuwang2011}.
Hereinafter we use $\phi$ instead of $\phi_{\eps}$ for simplicity in notations.


\subsection{Derivation of a modified Allen-Cahn equation by the Onsager principle}
We will use the Onsager principle to derive a gradient flow equation for the phase-field model~\eqref{e:phasefieldPb0}.
For that purpose, we define the energy dissipation function, which is the half of total energy dissipation rate,  as follows
\begin{align}\label{e:dissip}
\Phi=\frac{1}{2}\int_{\Omega}\dot{\phi}^2 \dx+\frac{\xi}{2}\int_{\Gamma}\dot{\phi}^2 \ds.
\end{align}
Here $\xi\geq 0$ is a phenomenological  parameter and we use $\dot{\phi}=\partial_t\phi$ for simplicity in notation.
The function $\Phi$ includes the dissipation in the bulk domain and also that on the boundary when $\xi>0$.
The Rayleighian is given by
\begin{equation}\label{e:Rayleign}
\mathcal{R}(\phi;\dot{\phi})=\Phi(\dot{\phi})+\dot{\mathcal{E}}_{\eps}(\phi;\dot{\phi}),
\end{equation}
 where $\dot{\mathcal{E}}_{\eps}$ is the changing rate of the total potential energy $\mathcal{E}_\eps$. By integration by part, we have
 \begin{equation}\label{e:rateEng}
\dot{\mathcal{E}}_{\eps}(\phi;\dot{\phi})=\int_{\Omega}(-\eps\Delta\phi+\eps^{-1} f'(\phi))\dot{\phi} \dx
+ \int_{\Gamma} (\eps \partial_n \phi+g'(\phi))\dot{\phi}
\ds + \int_{\partial\Omega\setminus\Gamma} \eps (\partial_n \phi)\dot{\phi}
\ds.
\end{equation}

By the Onsager principle, a dynamic equation for $\phi$ can {be} derived by minimizing the Rayleighian $\mathcal{R}(\phi;\dot{\phi})$ with respect to $\dot{\phi}$
under the constraint that $\int_{\Omega} \phi \dx =C_1$. Notice that this condition is equivalent to
$\int_{\Omega}\dot \phi \dx=0$ since the initial value  $\phi(x,0)$ satisfies $\int_{\Omega} \phi(x,0) \dx =C_1$. Therefore, the dynamic equation
is determined by the variational problem
\begin{equation}\label{e:Onsager0}
\min_{\int_{\Omega}\dot \phi \dx=0}\mathcal{R}(\phi;\dot{\phi}).
\end{equation}
  By introducing a Lagrangian multiplier $\lambda$,
we obtain a modified functional
$$\mathcal{R}_{\lambda}=\mathcal{R}-\lambda\int_{\Omega}\dot \phi \dx= \Phi+\dot{\mathcal{E}}_{\eps}-\lambda\int_{\Omega}\dot \phi \dx.$$
Direct computations for the Fr\'{e}chet derivative with respect to $\dot{\phi}$ give
\begin{align*}
\langle \frac{\delta \mathcal{R}_{\lambda}}{\delta\dot \phi},\psi \rangle &=\int_{\Omega}(\dot{\phi}-\eps\Delta\phi+\eps^{-1} f'(\phi)-\lambda)\psi \dx
+ \int_{\Gamma} (\xi\dot{\phi}+\eps \partial_n \phi+g'(\phi))\psi   \ds+ \int_{\partial \Omega\setminus\Gamma}  \eps(\partial_n \phi) \psi   \ds.
\end{align*}
The  Euler-Lagrange equation of \eqref{e:Onsager0}  is given by
\begin{equation}
\label{e:EL}
\left\{
\begin{array}{l}
\langle\frac{\delta \mathcal{R}_{\lambda}}{\delta\dot \phi},\psi \rangle=0,  \qquad\qquad \forall \psi;\\
\frac{\delta \mathcal{R}_{\lambda}}{\delta\lambda}=0;
\end{array}
\right.
\end{equation}
which leads to a modified Allen-Cahn equation,
\begin{equation}\label{e:AC}
\left\{
\begin{array}{ll}
\partial_t \phi -\eps\Delta\phi +\eps^{-1} f'(\phi)=\lambda, & \hbox{in }\Omega; \\
\xi\partial_t\phi+\eps\partial_n\phi +g'(\phi)=0, &\hbox{on }\Gamma;\\
\partial_n\phi =0, &\hbox{on }\partial\Omega\setminus\Gamma;\\
\int_{\Omega}\phi dx=C_1.&
\end{array}
\right.
\end{equation}
Here 
$f'(\phi)=\phi^3-\phi$ and $g'(\phi)={-\frac{3\sigma}{4}}\cos\theta_Y(1-\phi^2)=-\frac{\sqrt{2}}{2}(1-\phi^2)\cos\theta_Y$. Notice that when $\xi>0$, the second equation in \eqref{e:AC} is a dynamic relaxed boundary condition as in \cite{Qian:2003iz,chenWangXu2014}.
When $\xi=0$, it is a standard Neumann boundary condition.

By the above derivation, we can easily see the energy decay property of the solution of the Allen-Cahn equation~\eqref{e:AC}.
Suppose that $\phi$ is the solution of the modified   Allen-Cahn equation~\eqref{e:AC}.
By the equations~\eqref{e:AC}, \eqref{e:rateEng}, \eqref{e:EL} and $\int_{\Omega}\dot{\phi}dx=0$, we have
\begin{align*}
\frac{d \mathcal{E}_{\eps}}{d t}&=\int_{\Omega}(-\eps\Delta\phi+\eps^{-1} f'(\phi))\dot{\phi} \dx
+ \int_{\partial \Omega} (\eps \partial_n \phi+g'(\phi))\dot{\phi}  ds+ \int_{\partial \Omega\setminus\Gamma}  \eps(\partial_n \phi) \dot{\phi}   \ds\\
&=\int_{\Omega}(-\dot{\phi}+\lambda)\dot{\phi} \dx-\xi \int_{\Gamma}\dot{\phi}^2 \ds =-\int_{\Omega}\dot{\phi}^2 \dx-\xi \int_{\Gamma}\dot{\phi}^2 \ds\leq 0.
\end{align*}
Therefore, when $t$ goes to infinity, the stationary solution of the Allen-Cahn equation corresponds to a minimizer of the problem \eqref{e:phasefieldPb0}.

The modified equation~\eqref{e:AC} is more difficult to solve numerically than the standard Allen-Cahn equation, since
 we have a nonlinear boundary condition on the solid boundary in the second equation of  \eqref{e:AC}.
  The boundary condition implies that
the contact angle of the liquid  is relaxed to or equal to the local  Young's  angle. In addition, the last equation of \eqref{e:AC} provides
a non-local constraint for the phase-field function. 
\section{A diffusion generated motion method}
To obtain the equilibrium state of the wetting problem, one can solve the modified Allen-Cahn equation  \eqref{e:AC}  by some standard numerical methods(e.g. \cite{shen2010numerical,du2020phase}).
However, when the parameter $\eps$ is small, the equation is very difficult to solve.
Here we will not solve the equation directly.
Instead, we will use it to develop an efficient diffusion generated motion method
for the wetting problem~\eqref{e:sharpPb1}.
The key idea is to use the Onsager Principle as a tool to approximate the Allen-Cahn equation.
\subsection{Approximation of an Allen-Cahn equation by using the Onsager Principle.}

In this subsection, we  ignore the non-local constraint $\int_{\Omega}\phi dx=C_1$ in the phase field equation~\eqref{e:AC}.
The constraint comes from the volume conservation condition of the  wetting problem~\eqref{e:sharpPb1} and will
be considered in next subsection.
We  consider an Allen-Cahn equation
\begin{equation}\label{e:AC1}
\left\{
\begin{array}{ll}
\partial_t \phi -\eps\Delta\phi +\eps^{-1} f'(\phi)=0, & \hbox{in }\Omega; \\
\xi\partial_t\phi+\eps\partial_n\phi +g'(\phi)=0, &\hbox{on }\Gamma;\\
\partial_n \phi=0, & \hbox{on }\partial\Omega \setminus \Gamma.
\end{array}
\right.
\end{equation}
Similar to the analysis in the previous section, it is easy to see
that the equation \eqref{e:AC1} can be derived by using the Onsager principle
for the Reighleignian \eqref{e:Rayleign}.
In the following, we will use the Onsager principle to do some approximations for the equation \eqref{e:AC1}.

Since $\eps$ is  a small parameter, standard asymptotic analysis (c.f. \cite{caginalp1988dynamics,xuwang2011}) implies that the solution of \eqref{e:AC1} can be expanded with respect to $\eps$ as
\begin{equation*}
\phi(x,t)=\phi_0(x,t)+\eps \phi_1(x,t)+\cdots,
\end{equation*}
where  $\phi_0(x,t)=\tanh(\frac{d(x,t)}{\sqrt{2}\eps})$ and $d(x,t)$ is a  signed distance function
with respect to the zero level set  of $\phi_0(x,t)$ such that $d(x,t)>0$ when $\phi_0(x,t)>0$.
When we are interested only in the leading order approximation of $\phi$, we can assume
\begin{equation}
\label{e:expansion0}
\phi(x,t) \approx \phi_0(x,t)=\Psi \Big (\frac{d(x,t)}{\sqrt{2}\eps} \Big):=\tanh\Big(\frac{d(x,t)}{\sqrt{2}\eps}\Big).
\end{equation}
In this approximation, the only unknown is  the signed distance function $d(x,t)$
with respect to the zero level set of $\phi_0$.

We now derive a dynamic equation for the signed distance function $d(x,t)$ by using the Onsager Principle.
It is easy to see that $\Psi'(\cdot)=1-\Psi^2(\cdot)$. Direct calculations give
\begin{align*}
\nabla \phi &\approx \frac{1}{\sqrt{2}\eps}\Psi'\Big(\frac{d(x,t)}{\sqrt{2}\eps}\Big)\nabla d= \frac{1}{\sqrt{2}\eps}\Big(1-\Psi^2\Big(\frac{d(x,t)}{\sqrt{2}\eps}\Big)\Big) \nabla d= \frac{1}{\sqrt{2}\eps}(1-\phi_0^2) \nabla d, \\
\partial_t \phi &\approx \frac{1}{\sqrt{2}\eps}\Psi'\Big(\frac{d(x,t)}{\sqrt{2}\eps}\Big)\partial_t d= \frac{1}{\sqrt{2}\eps}
\Big(1-\Psi^2\Big(\frac{d(x,t)}{\sqrt{2}\eps}\Big)\Big) \partial_t d= \frac{1}{\sqrt{2}\eps}(1-\phi_0^2) \partial_t d.
\end{align*}
Combing the above approximations  with \eqref{e:expansion0}, the energy dissipation function (defined in \eqref{e:dissip}) in the system can be calculated  as
\begin{align}
\Phi &\approx \frac{1}{4\eps^{2}}\int_{\Omega}(1-\phi_0^2) ^2 (\partial_t d)^2 \dx+\frac{\xi}{4\eps^{2}}\int_{\Gamma}(1-\phi_0^2) ^2  (\partial_t d)^2 \ds\\
&=\frac{1}{\eps^2}\int_{\Omega} f(\phi_0) (\partial_t d)^2 \dx+\frac{\xi}{\eps^2}\int_{\Gamma} f(\phi_0) {(\partial_t d)}^2 \ds.
\end{align}
Similar calculations for the total free energy(defined in \eqref{e:diffuEng}) lead to
\begin{equation}
\mathcal{E}_{\eps} \approx \int_{\Omega}\frac{1}{\eps} f(\phi_0)  (|\nabla d(x,t) |^2 +1) \dx
+\int_{\Gamma} g(\phi_0) \ds.
\end{equation}
The time derivative of $\mathcal{E}_\eps$ is calculated by
\begin{align*}
\dot{\mathcal{E}}_{\eps} \approx &\int_{\Omega}\frac{2}{\eps} f(\phi_0)\nabla d \cdot \nabla \partial_t d   +\frac{(1+|\nabla d|^2)}{\eps}  f'(\phi_0)\Psi'\Big(\frac{d(x,t)}{\sqrt{2}\eps}\Big)  \frac{\partial_t d}{\sqrt{2}\eps}\dx+\int_{\Gamma} g'(\phi_0)\Psi'\Big(\frac{d(x,t)}{\sqrt{2}\eps}\Big)  \frac{\partial_t d}{\sqrt{2}\eps}\ds\\
=&\int_{\Omega}-\frac{2}{\eps} f(\phi_0 )\Delta d \partial_t d    +\frac{(1-|\nabla d|^2)}{\eps}  f'(\phi_0)\Psi'\Big(\frac{d(x,t)}{\sqrt{2}\eps}\Big)  \frac{\partial_t d}{\sqrt{2}\eps}\dx+ \int_{\partial\Omega\setminus\Gamma} \frac{2}{\eps} f(\phi_0)\partial_n d\partial_t d \ds \\
&+\int_{\Gamma}\frac{2}{\eps} f(\phi_0)\partial_n d\partial_t d + g'(\phi_0)(1-\phi_0^2)  \frac{\partial_t d}{\sqrt{2}\eps}\ds\\
=&-\int_{\Omega}\frac{2}{\eps} f(\phi_0) \Delta d \partial_t d \,  \dx+\int_{\Gamma}\frac{2}{\eps} f(\phi_0) (\partial_n d -\cos\theta_Y) {\partial_t d}\, \ds + \int_{\partial\Omega\setminus\Gamma} \frac{2}{\eps} f(\phi_0)\partial_n d\partial_t d \ds\\
\approx&-\int_{\Omega}\frac{2}{\eps} f(\phi_0) \Delta d \partial_t d \,  \dx+\int_{\Gamma}\frac{2}{\eps} f(\phi_0) (\partial_n d -\cos\theta_Y) {\partial_t d}\, \ds.
\end{align*}
In the derivations, we have used integration by part in the second equation, and  used
the facts that $|\nabla d|=1$,  $g'(\phi_0)={-\frac{\sqrt{2}}{2}\cos\theta_Y(1-\phi_{0}^2)}$
and $f(\phi_0)=\frac{(1-\phi_{0}^2)^2}{4}$ in the third equation. In the last equation, we ignore the integral term on $\partial\Omega\setminus\Gamma$, since $f(\phi_0)\approx f(-1)=0$ on  $\partial\Omega\setminus\Gamma$.
This can be seen from the equation \eqref{e:expansion0} when the liquid-vapor interface $\{x|\phi_0(x,t)=0\}$ is far from the
boundary $\partial\Omega\setminus \Gamma$.

The Rayleighian is  approximated by
\begin{align*}
\mathcal{R}&=\Phi+\dot{\mathcal{E}}_{\eps}\\
&\approx
\frac{1}{\eps^2}\int_{\Omega} f(\phi_0) \big((\partial_t d)^2-2\eps \Delta d \partial_t d\big) \dx +\frac{1}{\eps^2}\int_{\Gamma} f(\phi_0) (
{\xi}{\partial_t d} +2\eps(\partial_n d -\cos\theta_Y)) {\partial_t d}\, \ds.
\end{align*}
It is easy to see that Fr\'{e}chet derivative of $\mathcal{R}$ with respect to $\partial_t d$ is given by
\begin{align*}
\langle \frac{\delta \mathcal{R}}{\delta (\partial_t d)},\psi\rangle\approx &\int_{\Omega} \frac{2}{\eps^2} f(\phi_0)(\partial_t d- \eps\Delta d)\psi \,\dx +\int_{\Gamma} \frac{2}{\eps^2} f(\phi_0) (\xi\partial_t d+\eps \partial_n d -\eps\cos\theta_Y)\psi\, \ds
\end{align*}
By using the Onsager principle,
the dynamic equation for $d$ is obtained by minimizing the approximated Rayleighian with respect to $\partial_t d$.
This leads to a linear heat equation for $d$,
\begin{equation}\label{e:PDE1n}
\left\{
\begin{array}{ll}
\partial_t d -\eps\Delta d =0, & \hbox{in }\Omega; \\
\xi \partial_t d+\eps (\partial_n d -\cos\theta_Y)=0, &\hbox{on }\Gamma.\\
\end{array}
\right.
\end{equation}
We can see that when $\xi=0$, the  condition on $\Gamma$ is a standard Neumann boundary condition;
and when $\xi>0$, it is a linear dynamic boundary condition.

The equation \eqref{e:PDE1n} is much easier to solve than  \eqref{e:AC1} since it is a linear
equation with a linear boundary condition. However, the property
$|\nabla d|=1$ is not preserved by the solution of the heat equation~\eqref{e:PDE1n}. Therefore,
 $\tanh(\frac{d(x,t)}{\sqrt{2}\eps})$ is a good approximation
for the solution of \eqref{e:AC1}  only when $t$ is small. Nevertheless, it is enough for us
to derive a numerical scheme if we re-distance the solution of \eqref{e:PDE1n}
every time step.

\subsection{The numerical scheme}
We will construct a numerical method for  the wetting problem~\eqref{e:sharpPb1}. It includes the following three main steps:
 first to solve the linear diffusion equation \eqref{e:PDE1n}, then to re-distance its solution, thirdly to correct the volume of
 the liquid domain. More details
are described as follows.

{\it Solving a heat equation.} Given a signed distance function $d_k$, suppose that
 the liquid phase is given by $D_1^{k}=\{x:d_k(x)>0\}$ and such that $|D_1^{k}|=V_0$. We first solve
the following  heat equation until  time $\delta t$:
\begin{equation}\label{e:PDE1}
\left\{
\begin{array}{ll}
\partial_t \varphi -\Delta \varphi =0, & \hbox{in }\Omega; \\
\xi\partial_t \varphi+\partial_n \varphi =\cos\theta_Y, &\hbox{on }\Gamma;\\
\varphi(x,0)= d_k(x).
\end{array}
\right.
\end{equation}
The equation is equivalent to \eqref{e:PDE1n} after a  time scaling.
 In the above equation,
we have not fixed the  condition for $\varphi$ on the boundary $\partial \Omega\setminus\Gamma$. Since we assume
the boundary is far from the liquid-vapor interface, we could choose the boundary condition
freely. For example, we can choose a Dirichlet boundary condition $\varphi =d_k(x)$ or a Neumann boundary
condition $\partial_n \varphi =0$. In our simulations, we simply set $\partial_n \varphi =0$ on $\partial \Omega\setminus\Gamma$.

Since the equation \eqref{e:PDE1} is a linear diffusion equation, many standard numerical methods can be used.
In our simulations, 
we use the backward Euler scheme to discretize the time derivative  that $\partial_t \varphi=\frac{\varphi^{n+1}-\varphi^{n}}{\delta t}$,
where $\delta t$ is the time step size.
For the Laplace operator, we adopt the standard five point finite difference scheme to discretize it.
Namely, for a uniform partition of $\Omega=(x_0,x_1)\times(y_0,y_1)$ with mesh size $h$, for the grid points inside $\Omega$, we set
\begin{equation*}
\frac{\varphi^{n+1}_{i,j}-\varphi^{n}_{i,j}}{\delta t}-\frac{\varphi^{n+1}_{i-1,j}+\varphi^{n+1}_{i+1,j} +  \varphi^{n+1}_{i,j+1} + \varphi^{n+1}_{i,j-1}-4  \varphi^{n+1}_{i,j}}{h^2}=0.
\end{equation*}
On the lower boundary $\Gamma$, we discretize a dynamic boundary condition as follows,
\begin{equation*}
(1+\frac{2\xi}{h})\frac{\varphi^{n+1}_{i,1}-\varphi^{n}_{i,1}}{\delta t}-\frac{\varphi^{n+1}_{i-1,1}+\varphi^{n+1}_{i+1,1} + 2 \varphi^{n+1}_{i,2} -4  \varphi^{n+1}_{i,1}}{h^2}=\frac{2\cos\theta_Y}{h}.
\end{equation*}
It is a combination of the discretization of the linear diffusion equation and  the discretization of dynamic relaxed boundary conditions.
When $\xi=0$, this equation is reduced to the standard finite difference discretization of the Neumann boundary condition.

For simplicity in notations, we denote the process of solving the heat equation \eqref{e:PDE1}  by
\begin{equation*}
\varphi(x,\delta t)=\mathrm{EquSolver}(d_k,\delta t).
\end{equation*}

{\it The re-distance of the function $\varphi$.}
Since the solution $\varphi(x,\delta t)$ is not a signed distance function, we {need to} transform it into  a signed
distance function while keeping its zero level set unchanged.
Denote by $\tilde{d}_{k+1}$ the signed distance function with respect to $\varphi(x,\delta t)$, satisfying
\begin{equation*}
\nabla \tilde{d}_{k+1}=\frac{\nabla\varphi(x,\delta t)}{|\nabla\varphi(x,\delta t)|}
\end{equation*}
on the interface.
Computing $\tilde{d}_{k+1}$ from $\varphi(x,\delta t)$ is the standard re-distance process in
the level-set method.  Many
efficient algorithms have been developed(e.g. \cite{sethian1999level,russo2000remark,osher2001level,cheng2008redistancing,elsey2014fast}). Here we adopt the method based on a fast marching technique \cite{sethian1999level}.
We denote the re-distance process by
\begin{equation*}
  \tilde{d}_{k+1} = \text{Redist}(\varphi(x,\delta t)).
\end{equation*}

{\it Correction of the volume.}
Finally, notice that the volume of the domain surrounded by the zero level-set of $\tilde{d}_{k+1}$ is not
 equal to the initial value $V_0$. To preserve the volume of the liquid domain, we {need to} make
some correction for $\tilde{d}_{k+1}$ as in the {volume-preserving} threshold dynamics method \cite{ruuth2003simple}.
We first search for a constant $\delta^*$ such that the domain $|\{x|\tilde{d}_{k+1}(x)>\delta^*\}|=V_0$,
and then  set
$$d_{k+1}(x)=\tilde{d}_{k+1}(x)-\delta^*.$$
Generally speaking, to find a constant $\delta^*$ efficiently is not easy\cite{ruuth2003simple,kublik2011algorithms,Svadlenka14}.
In our numerical simulations, we use a simple bisection method. For that purpose, we must have
an interval containing $\delta^*$. For that purpose, we first  find an estimate of
$\delta^*$ using the method of counting the grid number as in \cite{xuWangWang2016}. Suppose
we have a uniform partition for $\Omega$ with mesh size $h$.
We sort the values $\tilde d_{k+1}(x_{ij})$ on the grid points into {a decreasing array} $\mathcal S$.
Then $\delta^*$ can be approximated roughly by a value $\hat{\delta}=(\mathcal{S}(M)+\mathcal{S}(M+1))/2$, where
$M=\lfloor V_0/h^2\rfloor$, the integer part of $V_0/h^2$.
We then easily extend it into a small interval  which contains $\delta^*$.


Repeat the above three steps iteratively.
We get a series of signed distance function $d_k$, $k=0,1,\cdots$. Each $d_k$ corresponds to
a liquid domain $D^k_1$ satisfying $|D^k_1|=V_0$. We expect that $\{D^k_1,D^k_2\}$ with $D^k_2:=\Omega\setminus D_1^k$
gives an energy decreasing sequence  for the problem~\eqref{e:sharpPb1}. (The property will be studied in next section.)
This leads to the following diffusion generated motion algorithm for the wetting problem.

\begin{algorithm}[H]
\textbf{Algorithm 1}
\begin{algorithmic}
\State \textbf{Step 0}. Give an initial signed distance function $d_0$ such that $|\{x|{d}_0(x)>0\}|=V_0$, a {constant tolerance} $TOL$, and set $k=0$.
\State \textbf{Step 1}.  Solve the heat equation \eqref{e:PDE1} with initial function $d_k$ until a time $\delta t$   (only one implicit Euler step) to obtain
\begin{equation}\label{ss1}
\varphi(x,\delta t)=\mathrm{EquSolver}(d_k,\delta t).
\end{equation}
\State \textbf{Step 2}. Construct the signed distance function $\tilde{d}_{k+1}$ from $\varphi(x,\delta t)$\\
\begin{equation}\label{ss2}
  \tilde{d}_{k+1} = \text{Redist}(\varphi(x,\delta t)).
\end{equation}
\State \textbf{Step 3}. Find a constant $\delta^*$ such that the domain $|\{x|\tilde{d}_k(x)>\delta^*\}|=V_0$.
Set
$$d_{k+1}(x)=\tilde{d}_{k+1}(x)-\delta^*.$$
\State \textbf{Step 4}. If $\|d_{k+1}-d_{k}\|<TOL$, stop. Otherwise, set $k=k+1$ and go back to \textbf{Step 1}.
\end{algorithmic}
\end{algorithm}\label{Algorithm1}

The algorithm is much more efficient
than directly solving the modified Allen-Cahn equation \eqref{e:AC}, since we avoid to resolve the
very  thin inner layer induced by a small parameter $\eps$ and also avoid the nonlinear boundary condition on $\Gamma$.
Notice that the parameter $\xi$ in the dynamic boundary condition in \eqref{e:PDE1} can be regarded as an adjustable
parameter when we are only interested in the stationary state of wetting problems.
Finally, if the boundary condition is replaced by a periodic boundary condition, the method is
reduced to the standard  method in \cite{esedog2010diffusion} for mean curvature flows.
Here the main difference is that we use a dynamic boundary condition, which is used
to describe the contact angle condition in wetting problems.

\section{Energy stability analysis}
We  study the energy stability property for Algorithm 1 in this section.
As mentioned above, for any given $d_0$, which corresponds to a liquid domain $D^0_1=\{x|d_0(x)>0\}$,
the algorithm generates a sequence of $d_k$. Each $d_k$ corresponds to a liquid domain $D^k_1=\{x|d_k(x)>0\}$.
This implies that Algorithm 1 generates a (discrete) evolution of a liquid domain.
We will prove the total wetting energy \eqref{e:sharpEng1} corresponding to the liquid domain $D^k_1$(by setting $D^k_2=\Omega\setminus D_1^k$) decays
when $k$ increases.
For that purpose, we will first reformulate the main processes in Algorithm 1 into some geometric flows
for the surface of the liquid domain. Then we will prove some energy decay property for the geometric flows.
Hereinafter we restrict our analysis in two dimensions. For simplicity, we assume the liquid domain
$D_1^k$ is simply connected and we consider only the case $\xi=0$ in \eqref{e:PDE1}, which corresponds to a standard Neumann boundary condition:
\begin{equation}\label{e:Neumann}
\partial_n\varphi =\cos\theta_Y.
\end{equation}

\subsection{The approximate geometric flows for Algorithm 1}
We first consider Step 1 in Algorithm 1. When $\delta t$ is small and $\xi=0$, the zero level set of the solution of the linear diffusion equation \eqref{e:PDE1} approximates a mean curvature flow problem starting from
a curve $\{ x\in \Omega\, |\, d_k(x)=0\}$ with the contact angle being given by {Young's} angle $\theta_Y$.
Here a mean curvature flow  denotes a geometric flow for a curve which evolves with the normal velocity
equal to its local (mean) curvature.
When the zero level set of $\varphi$ has no intersection with the solid boundary $\Gamma$,
the relation between the solution of the heat equation and the mean curvature flow has been shown in \cite{esedog2010diffusion}. The approximation error is of order $O(\delta t^2)$(c.f. Equation (66)-(67) in \cite{esedog2010diffusion}).
The analysis still works for \eqref{e:PDE1} for the inner points on the zero level set of $\varphi(x,\delta t)$.
We need only to study the contact angle  condition. On
the zero level set of $\varphi(x,\delta t)$, its normal direction is given by $\frac{\nabla \varphi}{|\nabla \varphi|}$.
The contact angle $\theta $ between the zero level set of $\varphi(x,\delta t)$ and the solid boundary $\Gamma$ satisfies
\begin{equation}\label{e:apprAng}
\cos\theta=\frac{\nabla \varphi}{|\nabla \varphi|}\cdot\mathbf{n}=\frac{1}{|\nabla\varphi|}\partial_n\varphi=\frac{1}{|\nabla\varphi|}\cos\theta_Y.
\end{equation}
Here we have used the boundary condition for $\varphi$ on $\Gamma$.
The analysis(c.f. Equation (30) in \cite{esedog2010diffusion}) shows that
$$|\nabla\varphi(x,\delta t)|\approx|\nabla \varphi(x,0)|+O(\delta t)=1+O(\delta t),$$
in the vicinity of the zero level set of $\varphi$. Therefore \eqref{e:apprAng} implies that
$\theta\approx\theta_Y+O(\delta t)$.
In all, Step 1 of Algorithm 1 leads to a mean curvature flow
for the surface of the liquid domain with a contact angle  $\theta\approx \theta_Y$.

For Step 2 in Algorithm 1, i.e. the re-distance step, the zero level-set of $\tilde{d}_{k+1}$ is the same as that for $\varphi(x,\delta t)$.
 This implies {the} liquid domain is unchanged in this step.
Then we need only further consider Step 3 (the volume correction step) in Algorithm 1.
Since $\tilde{d}_{k+1}$ is a signed distance function,
it is easy to see that the volume correction step corresponds to a motion of the zero level-set of  $\tilde{d}_{k+1}$ with a constant normal velocity. More precisely, the zero level-set of $d_{k+1}$ is obtained by moving the zero level-set of $\tilde{d}_{k+1}$ in a constant normal velocity equal to $-1$ until a time $\delta^*$.

By the above analysis, one iteration of Algorithm 1 can be approximated by two geometric flows (a mean curvature flow and a constant normal velocity flow) for the surface of a liquid domain. For convenience in later analysis,
we  characterize the two geometric flows more precisely as follows.
\begin{itemize}
\item {\bf Process 1.} At $t=t_0=0$, $d_k$ gives a  droplet with a free boundary curve $\mathcal{L}_0$(as shown in Figure~\ref{curve_bc}).
The boundary $\mathcal{L}(t):=\{\vct X(s,t): 0\leq s \leq L(t) \}$ of the droplet evolves under a mean curvature flow in the sense that $v_n=\kappa$ and the
boundary condition that $\theta_0=\theta_1=\theta_Y$. Here $s$ is the arc-length parameter, $L(t)$ is the length
of the curve $\mathcal{L}(t)$ at time $t$, and $\theta_i$ with $i=0,1$ are  the contact angles
at the two ends of $\mathcal{L}(t)$.
This process ends until $t_1=t_0+\delta t$ and then we denote $\mathcal{L}_1=\mathcal{L}( t_1)$.
In this process, the volume of the droplet is denoted by $A(t)$ which decreases with time.  We set $A_0=A(t_0)=V_0$ and $A_1=A(t_1)$.
\item {\bf Process 2.} Starting from $\mathcal{L}_1$ at $t=t_1$, the boundary $\hat{\mathcal{L}}(t):=\{\hat{\vct X}(s,t): 0<s<\hat{L}(t)\}$ of the droplet moves under a constant velocity flow with $v_n=-1$. Here $s$ is the arc-length parameter
 and $\hat{L}(t)$ is the length of the curve $\hat{\mathcal{L}}(t)$ at time $t$. Notice that we have set $\hat{\mathcal{L}}(t_1)=\mathcal{L}_1$.
If we denote the volume of the droplet as ${A}(t)$ in the process, it is an increasing function with respect to time $t$ such that ${A}(t_1)=A_1$.
The process ends until $t_2=t_1+\delta^*$ such that the liquid volume corresponding to  $\hat{\mathcal{L}}(t_2)$ is equal to that to $\mathcal{L}_0$, namely ${A}(t_2)=A_0=V_0$.
\end{itemize}
\begin{figure}
  \centering
  \includegraphics[width=8cm]{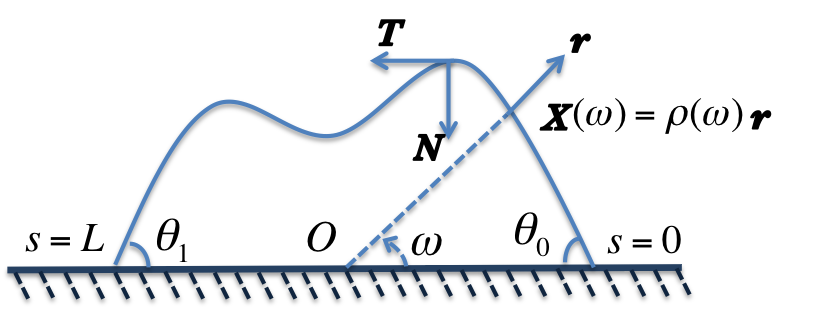}\\
  \caption{The evolving curve with constraints}\label{curve_bc}
\end{figure}

Here we use different notations $\mathcal{L}$ and $\hat{\mathcal{L}}$ for the curve of a droplet in the two different processes.
We will show the total energy decreases after an iteration of Process 1 and Process 2. In the following proof, we further assume the contact angle
of $\hat{\mathcal{L}}(t)$ in Process 2 are also keep fixed so that $\theta_{i}=\theta_Y$ for $i=0,1$. 
 This can be guaranteed when the contact angle is smaller than $90^o$
under {a proper} extension of the level-set function under the solid boundary when we do the {re-distance} step. Otherwise,
the contact angle may deviate slightly from $\theta_Y$ when $\delta^*$ is small.


\subsection{Energy decay property for the geometric flows}
In the following, we will analyse the energy decay property of an iteration of Process 1 and Process 2.
This is based on a careful analysis {of} the properties of the corresponding geometric flows in
the two processes. We  first introduce some results for a general geometric flow for a
 curve with both ends on a substrate.

As shown in Figure~\ref{curve_bc},  we consider a smooth curve on a substrate, which represents the surface of a liquid droplet.
The liquid domain is simply connected so that the curve does not intersect itself. For simplicity,
we assume the curve can be represented by a function in a polar coordinate as $\{\vct X(\omega):0\leq \omega\leq \pi \}$.
The curve  evolves with time and {can be} noted by $\{\vct X (\omega, t)\}$. Its length is given by $ L(t)$.
Notice that the curve can also be represented in an arclength parameter $s$ that $\{\vct X(s,t):0\leq s \leq L(t) \}$.
Then we have $\vct T=\partial_s \vct{X}=$ and $\partial_s \vct T=\kappa \vct N$.
Here $\vct{N}$ and $\vct{T}$ are respectively the unit normal vector and the unit tangential vector on the curve. $\vct N$ is obtained by rotating $\vct T$
for $90^o$ degrees in {the} counter-clockwise direction.
 $\kappa(s)$ is the curvature of
the curve, which is positive when the center of the circle of curvature is in the same direction of $\vct{N}$.
In addition, we have the relation that $\partial_{\omega} \vct X=|\partial_{\omega} \vct X|\vct{T}$
and $\rmd \omega= |\partial_{\omega} \vct X| \rmd s$. Here $|\cdot|$ is the norm of a vector in $\mathbb{R}^2$.

We use $ v_n $ and $ v_s $ to represent the radial and tangential components of velocity respectively. Then the velocity can be written as
\begin{equation*}
  \vct X_t = \vct v = \vct v_n + \vct v_s = v_n \vct N+v_s \vct T.
\end{equation*}
 We are interested in how the length of the curve and the area of the liquid domain (enclosed by the curve and substrate) change when the curve evolves with time.
This is given by the following lemma,
which is adapted from the result for closed curves(see Section 3.2  in \cite{kimmel2003numerical}).
\begin{lemma}\label{lemma_bc}
For an  evolving curve as shown Figure~\ref{curve_bc}, let ${L}(t)$ be the length of the  curve,
and ${A}(t)$ be the area enclosed by the curve and the lower  substrate. Then, we have
\begin{align}\label{e:Lt_ncl}
\frac{\rmd  L(t)}{\rmd t}&=-\int_{0}^{L} \kappa v_n \rmd s+v_x(0)\cos\theta_0- v_x(L)\cos\theta_1,\\
\label{e:At_ncl}
\frac{\rmd  A(t)}{\rmd t}&=-\int_{0}^{L} v_n \rmd s.
\end{align}
where $\theta_0$ and $\theta_1$ are respectively the contact angles at the points corresponding to $s=0$ and $s=L(t)$; $v_x(0)$ and $v_x(L)$
are respectively the {velocities} of the {two end} points in $x$ direction.
\end{lemma}

\begin{proof}
For proof of (\ref{e:Lt_ncl}) we argue as follows. Here we use $\left <\cdot,\cdot\right >$ to donate inner product of the vector in $\mathbb{R}^2$.
\begin{equation*}
\begin{aligned}
 \frac{\rmd}{\rmd t} {L}(t) &=\frac{\rmd }{\rmd t} \int_{0}^{\pi}\left\langle \partial_{\omega } \vct X, \partial_{\omega} \vct X\right\rangle^{1 / 2} \rmd \omega
 =\int_{0}^{\pi } \frac{\left\langle \partial_{t \omega}\vct X, \partial_{\omega}\vct X\right\rangle}{\left\langle \partial_{\omega} \vct  X,\partial _{\omega}\vct X\right\rangle^{1 / 2}} \rmd \omega \\
   &=\int_{0}^{\pi}\left\langle\frac{\partial}{\partial \omega}(v_n \vct{N}+v_s \vct T), \vct{T}\right\rangle \rmd \omega
   =\int_{0}^{\pi}\left\langle\left|\partial_{\omega}\vct X\right| \frac{\partial}{\partial s}(v_n \vct{N}+v_s \vct T), \vct{T}\right\rangle \rmd \omega \\ &=\int_{0}^{L(t)}\left\langle\frac{\partial}{\partial s}(v_n \vct{N}+v_s \vct T), \vct{T}\right\rangle \rmd s\\
   &=-\int_{0}^{L(t)} v_n \kappa \rmd s + \int_{0}^{L(t)} \partial_s v_{s} \rmd s =-\int_{0}^{L(t)} v_n \kappa \rmd s + v_s(L)-v_s(0).
 \end{aligned}
\end{equation*}
In the last second derivation, we have used the geometric equation that $\partial_s\vct N=-\kappa \vct T$.
Since the {two end} points can only move  in  horizontal direction, we have the constraint
\begin{equation}
v_s(L)=-v_x(L)\cos\theta_1, \qquad v_s(0)=-v_x(0)\cos\theta_0.
\end{equation}
Then we have proved the equation (\ref{e:Lt_ncl}).

The proof of (\ref{e:At_ncl}) can  be  done as follows. Notice that the curve can also be written as
$\vct X(\omega,t)=\rho(\omega,t)\vct r(\omega)$,
where $\vct r(\omega)$ is the unit vector in {the} radial direction.
Then we have
$$\partial_\omega \vct X(\omega,t)=\partial_{\omega} \rho(\omega,t) \vct r(\omega)+\rho(\omega,t) \vec{\theta}(\omega),$$
 where $\vec{\theta}(\omega)$ is the unit vector in the angular direction.
Then we have
$$
\partial_{t\omega} \vct X(\omega,t)=\partial_{t\omega} \rho(\omega,t) \vct r(\omega)+\partial_t \rho(\omega,t) \vec{\theta}(\omega).
$$
We define an outer product for vectors  $\vct a=(a_1,a_2)^T,\vct b=(b_1,b_2)^T \in \mathbb{R}^2$ (see Page~22 in \cite{kimmel2003numerical}) as
$$\left[ \vct a,\vct b\right ]:=\det\left (\begin{array}{cc}
a_1 & b_1\\
a_2 & b_2
\end{array}
\right)=a_1b_2-a_2b_1,
$$
which is a bilinear operator for the two vectors.
This leads to
$$
\left[ \vct X(\omega,t), \partial_{t\omega} \vct X(\omega,t) \right]=\rho(\omega,t)\partial_t \rho(\omega,t)\left[  \vct r(\omega) , \vec{\theta}(\omega)\right]
=\rho(\omega,t)\partial_t \rho(\omega,t).
$$
Then we have
\begin{align*}
 \frac{\rmd}{\rmd t} {A}(t) &=\frac{\rmd }{\rmd t} \int_{0}^{\pi}\frac{1}{2}( \rho(\omega,t))^2\rmd \omega
 =\int_{0}^{\pi} \rho(\omega, t)\partial_t \rho(\omega,t) \rmd \omega\\
 & =\int_{0}^{\pi}\left[ \vct X(\omega,t), \partial_{t\omega} \vct X(\omega,t)\right] \rmd \omega =\int_{0}^{\pi}\left[ \vct X(\omega,t), \partial_{\omega} \vct v(\omega,t)\right] \rmd \omega \\
 &=\left[ \vct X(\omega,t),  \vct v(\omega,t)\right]|_0^{\pi}-\int_{0}^{\pi}\left[ \partial_{\omega} \vct X(\omega,t), \vct v(\omega,t)\right] \rmd \omega\\
 &=-\int_{0}^{\pi}\left[|\partial_{\omega}\vct X|\vct T , \vct v(\omega,t)\right] \rmd \omega =-\int_{0}^{L(t)}\left[\vct T , \vct v(\omega,t)\right] \rmd s\\
 &= -\int_{0}^{L(t)} v_n \ds.
\end{align*}
In the derivation, we have used the integration by part, the fact that all the vectors $\vct v(0)$, $\vct v(\pi)$, $\vct X(0)$ and $\vct X(\pi)$
are in the horizontal direction, $\left[ \vct T, \vct T\right] =0$ and $\left[ \vct T, \vct N\right] =1$. We finish the proof of the lemma.
\end{proof}

In  wetting problems, we mainly interested in the total surface energy in a liquid-vapor-solid system, which is defined in \eqref{e:sharpEng1}.
For a droplet as in Figure~\ref{curve_bc}, the wetting energy can be rewritten as
\begin{equation*}
  \mathcal E(t) = \gamma\left ( L(t) -  |\vct X(0)-\vct X(L)|\cos \theta_Y\right ) +C.
\end{equation*}
Here $\gamma=\gamma_{LV}$ and $C=C_0+\frac{\gamma}{2}|\Gamma|\cos\theta_Y$. Here we assume that the contact angle $\theta_Y$ is a constant on the solid surface $\Gamma$.
Applying Lemma \ref{lemma_bc} directly yields,
\begin{lemma}\label{lemma_E}
For a droplet with an evolving surface as shown in Figure~\ref{curve_bc}, the total wetting energy satisfies
\begin{equation}\label{e:Et}
\frac{\rmd \mathcal E(t)}{\rmd t}=-\gamma \int_{0}^{L(t)} \kappa v_n \rmd s+\gamma v_x(0)(\cos\theta_0-\cos\theta_Y)-\gamma v_x(L)(\cos\theta_1-\cos\theta_Y),
\end{equation}
\end{lemma}
\begin{proof}
Apply Lemma \ref{lemma_bc} directly to get
\begin{equation*}
\begin{aligned}
   \frac{\rmd}{\rmd t}  L(t) -\frac{\rmd}{\rmd t}  |\vct X(0)-\vct X(L)|  \cos \theta_Y
 &=\left(-\int_{0}^{L} \kappa v_n \rmd s+v_x(0)\cos\theta_0-v_x(L)\cos\theta_1\right)-(v_x(0)-v_x(L))\cos\theta_Y\\
 &=-\int_{0}^{L} \kappa v_n \rmd s+v_x(0)(\cos\theta_0-\cos\theta_Y)-v_x(L)(\cos\theta_1-\cos\theta_Y).
 \end{aligned}
\end{equation*}
\end{proof}

Gathering the results above readily gives the following properties for the geometric flows in Process 1 and Process 2 defined in Section 5.1.
\begin{lemma}\label{corollary_E}
For the geometric flow in  Process 1, in which the normal velocity is equal to the mean curvature($v_n = \kappa$)  and the contact angle is given by Young's angle,
we have
\begin{align}\label{e:mkf_Et}
\frac{\rmd  \mathcal E}{\rmd t} &=-\gamma  \int_0^{L(t)} \kappa^2 \ds,\\
\label{e:mkf_At_E}
   \frac{\rmd A}{\rmd t} &= -2\theta_Y.
\end{align}
For the geometric flow in Process 2, in which the normal velocity is a constant($v_n=-1$), suppose the contact angle is also equal to {Young's} angle  $\theta_Y$,  then we have
\begin{align}\label{e:constf_E}
\frac{\rmd  {\mathcal E}}{\rmd t}   &=
 2\gamma \theta_Y, \\
\label{e:constf_At_E}
   \frac{\rmd {A}}{\rmd t} &= \hat {L}(t).
\end{align}
\end{lemma}
\begin{proof}
The lemma is directly from the previous two lemmas by noticing that $\int_0^{L(t)}\kappa \ds=2\theta_Y$
and $\int_0^{\hat{L}(t)}\kappa \ds=2\theta_Y$ in the two processes.
\end{proof}

We are  ready to prove the energy decay property of an iteration of Process 1 and Process 2. For clarity in presentation, we introduce a few more notations. Let $A_0=A(t_0)$, $A_1=A( t_1)$ and $A_2={A}(t_2)$ be the volumes(areas) of the droplet at the beginning of
Process 1, at the end of Process 1(or equivalently the beginning of Process 2) and {the end} of Process 2, respectively. Similarly,
$\mathcal E_i, i=0,1,2$ are respectively the total wetting energies in the corresponding states.

\begin{theorem}
 For a droplet as shown in Figure \ref{curve_bc}, it evolves  by a mean curvature flow as described in Process 1 {in the first stage} ($t_0<t< t_1$), and then correct its volume to initial value $ A_0$ by a constant velocity flow as described in Process 2 in the second stage ($t_1<t< t_2$). If we assume that $\mathcal{L}(t_1)$ is not a circular curve,
 then   the total wetting energy decreases after an iteration of Process 1 and Process 2, namely
\begin{equation}\label{e:energydecay}
  \mathcal E_2  \leq \mathcal E_0,
\end{equation}
when $\delta t$ is small enough.
\end{theorem}
\begin{proof}
By Lemma \ref{corollary_E}, we obtain the relation of the wetting energies,
\begin{equation}\label{e:E0_1}
  \mathcal E_0 - \mathcal E_1 = \gamma \int_{t_0}^{t_1}\int_0^{L(t)} \kappa^2 \rmd s \rmd t = \gamma \int_{t_1-\delta t}^{t_1}\int_0^{L(t)} \kappa^2 \rmd s \rmd t =:I,
\end{equation}
\begin{equation}\label{e:E2_1}
  \mathcal E_2 - \mathcal E_1 =2\gamma \theta_Y (t_2-t_1)= 2\gamma \theta_Y \delta^*=:II.
\end{equation}
We need only to prove
\begin{equation}\label{e:energydecay1}
 II \leq I.
\end{equation}

For the volume of the droplet, we have
\begin{align}
&   A_0 -  A_1 =2\theta_Y\delta t, \label{e:A0_1}\\
 &  A_2 -  A_1=\int_{t_1}^{t_2} \hat{L}(t) \rmd t. \label{e:A2_1}
\end{align}
  From the proof of Lemma 2, we know that, for the constant velocity flow in Process 2,
\begin{equation}
 \frac{d \hat{{L}}(t) }{d t}=2\theta_Y+(v_x(0)-v_x(\hat{L}))\cos\theta_Y.\label{e:temp0}
\end{equation}
Notice that $v_x(0)\sin\theta_Y=-v_n=1$ and $v_x(L)\sin\theta_Y=v_n=-1$.
So we have $ \frac{d \hat{{L}}(t) }{d t}=2(\theta_Y+\mathrm{ctan}\,\theta_Y).$
This leads to $\hat{{L}}(t)-\hat{{L}}(t_1)=2(\theta_Y+\mathrm{ctan}\,\theta_Y)(t-t_1) .$
Then the right hand side term of the equation~\eqref{e:A2_1} can be computed as
\begin{align*}
\int_{t_1}^{t_2} \hat{{L}}(t) \rmd t
&=\delta^*\big(\hat{{L}}(t_1)+(\theta_Y+\mathrm{ctan}\,\theta_Y) \delta^*\big).
\end{align*}
Due to the volume correction condition that ${A}_0={A}_2$, we have
\begin{equation}\label{e:E3}
2\theta_Y\delta t =\delta^*\big(\hat{{L}}(t_1)+(\theta_Y+\mathrm{ctan}\,\theta_Y) \delta^*\big).
\end{equation}
This leads to  $\delta^*= \frac{\sqrt{\hat{L}(t_1)^2+8\theta_Y(\theta_Y+\mathrm{ctan}\,\theta_Y) \delta t}-\hat{L}(t_1)}{2(\theta_Y+\mathrm{ctan}\,\theta_Y) }$.
Notice that
 $\hat{{L}}(t_1)$ is the same as ${L}( t_1)$. We have
 \begin{equation}
\delta^*= \frac{\sqrt{{L}(t_1)^2+8\theta_Y(\theta_Y+\mathrm{ctan}\,\theta_Y) \delta t}-{L}(t_1)}{2(\theta_Y+\mathrm{ctan}\,\theta_Y) }.
 \end{equation}
Then we have
\begin{align}
II=2\gamma \theta_Y \delta^*&= \frac{\gamma \theta_Y }{\theta_Y+\mathrm{ctan}\,\theta_Y}
\big(\sqrt{{L}(t_1)^2+8\theta_Y(\theta_Y+\mathrm{ctan}\,\theta_Y) \delta t}-{L}(t_1)\big) \nonumber \\
&=\gamma \int_{t_1}^{ t_1+\delta t}\frac{4\theta_Y^2 }{\sqrt{{L}(t_1)^2+8\theta_Y(\theta_Y+\mathrm{ctan}\,\theta_Y)(t-t_1)}} \rmd t\nonumber \\
&=\gamma \int_{t_1-\delta t}^{ t_1}\frac{4\theta_Y^2 }{\sqrt{{L}(t_1)^2+8\theta_Y(\theta_Y+\mathrm{ctan}\,\theta_Y)(t_1-t)}} \rmd t
=: \gamma\int_{t_1-\delta t}^{ t_1}\tilde{g}(t) \rmd t
\label{e:termII}
\end{align}

It is easy to see that
$$\tilde g(t_1)=\frac{4\theta_Y^2}{L( t_1)}.$$
By Cauchy-Buniakowsky-Schwarz inequality, we have
\begin{equation*}
2\theta_Y=\int_{0}^{L(t_1)}\kappa \ds \leq \Big(\int_{0}^{L(t_1)}\kappa^2 \ds \Big)^{1/2} (L(t_1))^{1/2}.
\end{equation*}
The equality holds only when $\kappa$ is a constant on $\mathcal{L}(t_1)$. This implies
\begin{equation}
\tilde g( t_1)<\int_0^{L( t_1)}\kappa^2\rmd s.
\end{equation}
when $\mathcal{L}(t_1)$ is not a circular curve.
 Since $\tilde g(t)$ is a smooth function with
respect to $t$, we then have
\begin{equation}
\int_{t_1-\delta t}^{t_1} \tilde g(t) dt \leq \int_{t_1-\delta t}^{t_1}  \int_0^{L( t)}\kappa^2\rmd s.
\end{equation}
when $\delta t$ is small enough. This ends the proof of the theorem.
\end{proof}


\section{Numerical examples}

\subsection{Accuracy check}
In this subsection, we will do  numerical experiments to test the convergence rate with respect to $\delta t$ for the diffusion generated
motion method(Algorithm 1).
For that purpose, we consider a liquid drop on a homogeneous planar surface.
It is known that the  problem~\eqref{e:sharpPb1} has an explicit solution.
It corresponds to a circular droplet with the contact angle equal to {Young's} angle on the substrate.
Therefore we can compute the stationary profile of a liquid drop for various Young's angles.

We do tests for several choices of {Young's} angle $\theta_Y$.
In all these tests, we set $\Omega=[0,1]^2$.
We put the initial curve in the middle of the lower boundary of $\Omega$.
When {Young's} angle $\theta_Y$ is not equal to $90 \degree$,
we set the surface of the initial droplet as a semicircle with the radius $R_0 = 0.25$.
When $ \theta_Y = 90 \degree $, the initial droplet is a circular domain with an initial volume $V_0 = \pi/16$
and with two contact points given by $(3/8,0)$ and $(5/8,0)$.
We set $TOL=10^{-10}$ in the algorithm. We find that {Algorithm 1} always converges in the sense that
the tolerance is obtained after some iterations.
In each experiment, we use a uniform $N\times N$ mesh and fix a time step $\delta t$(as listed in Table~\ref{110a}-\ref{70a}), where we set a unit time as $ t_{unit} = 1/64$.
We compute the  error in the following way,
\begin{equation*}
  {Err}_{\infty} := \operatorname{dist}(\mathcal{L}_k,\mathcal{L} )=\sup_{\mathrm{x}\in \mathcal{L}_k} \inf_{\mathrm{y}\in \mathcal{L} }\|\mathrm{x}-\mathrm{y}\|
\end{equation*}
Here $\mathcal{L}_k$ represents the zero level set of the numerical solution $d_k$, $\mathcal{L} $ is
 the circular curve corresponding to the exact solution.
\begin{figure}[!h]
 \centering
  \begin{minipage}[t]{0.39\linewidth}
    \centering
  \includegraphics[width=7cm]{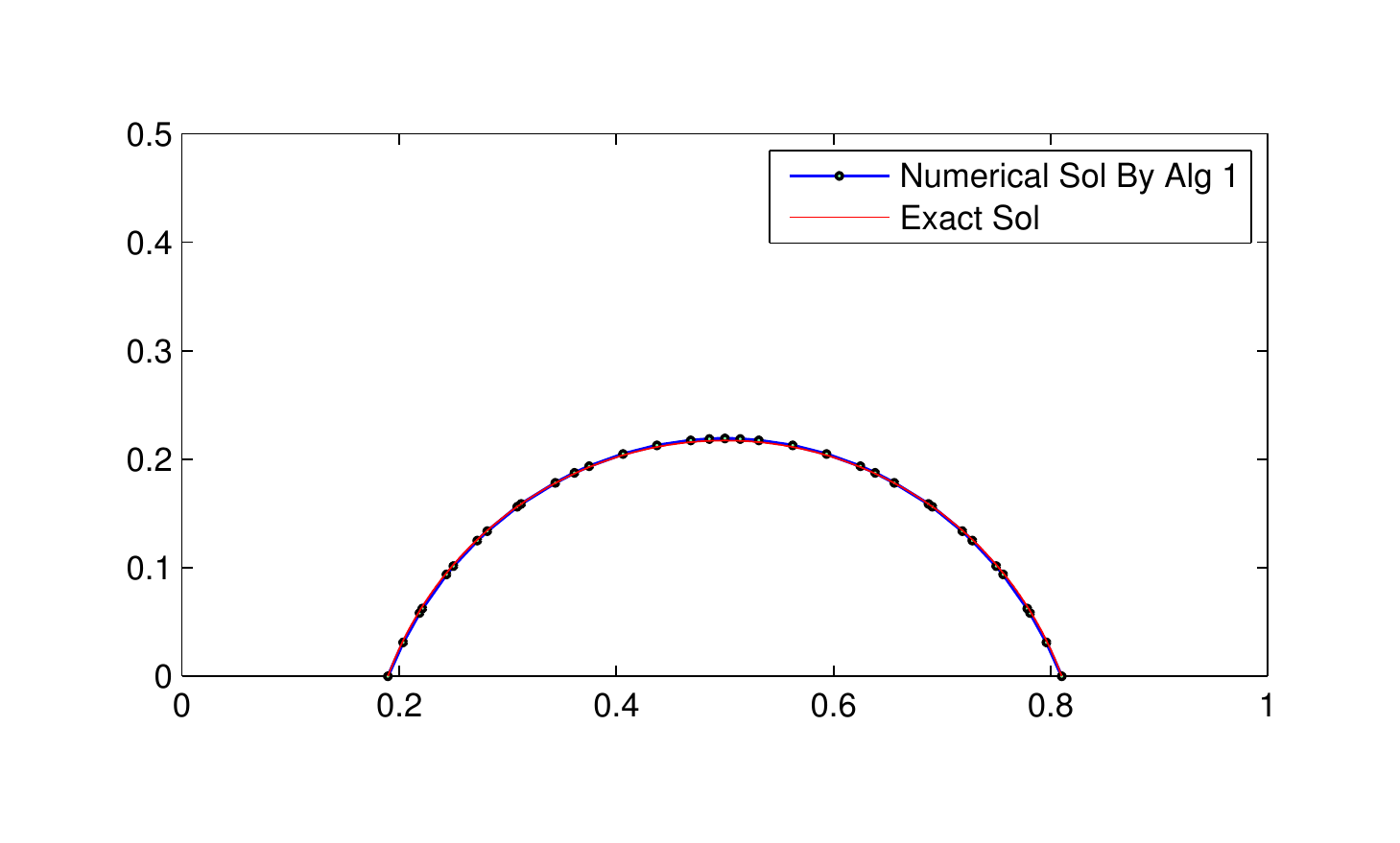}
  \vspace{-1.2cm}
    \caption*{(a)$\theta = 70^o, N=32$.}
    \end{minipage}
      \begin{minipage}[t]{0.39\linewidth}
    \centering
   \includegraphics[width=7cm]{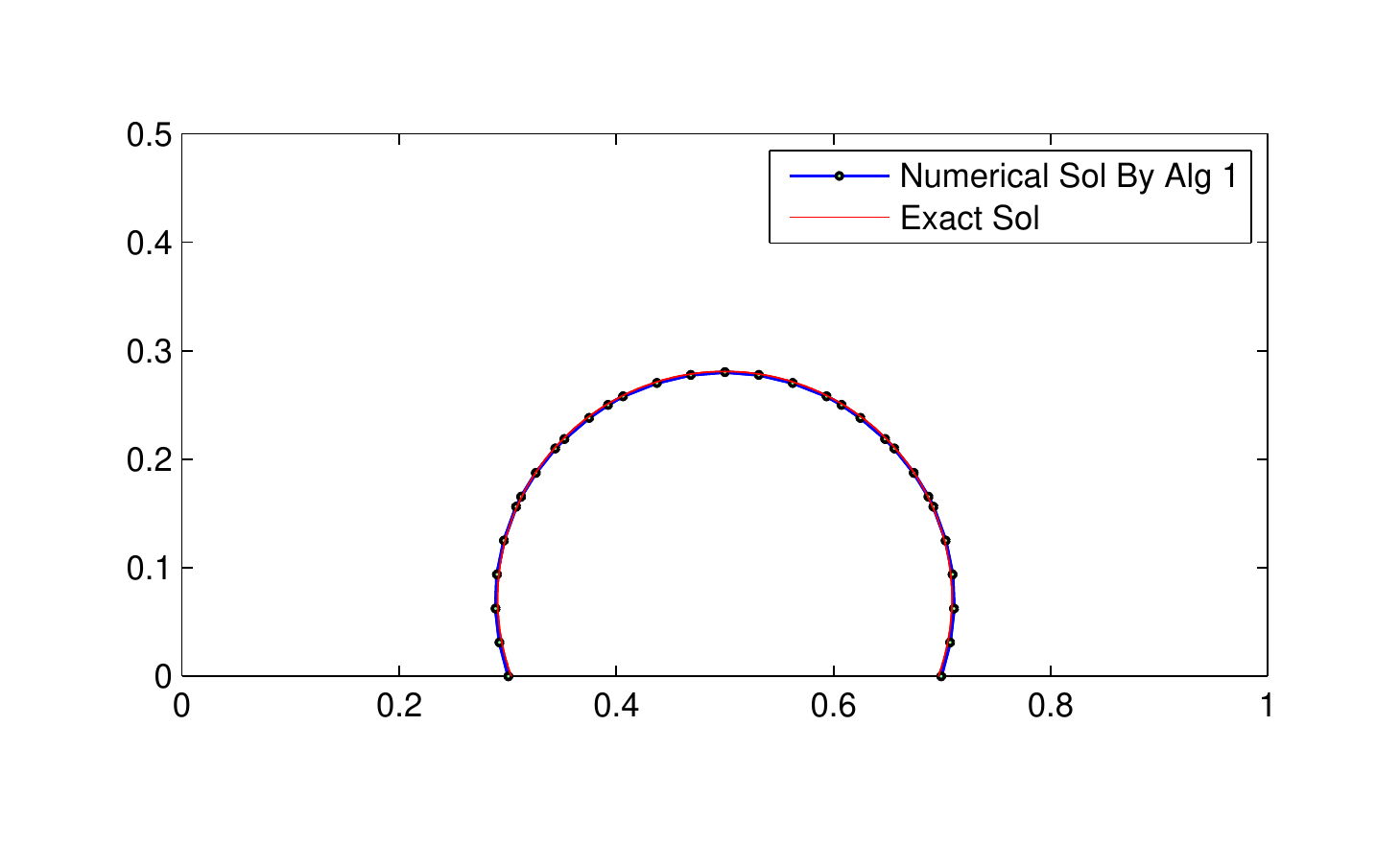}
     \vspace{-1.2cm}
      \caption*{(b)$\theta = 110^o, N=32$.}
    \end{minipage}
  \caption{A comparison between the numerical solution obtained by Algorithm 1 and the exact solution.
   }\label{nmrc}
\end{figure}

Figure~\ref{nmrc} shows  numerical solutions and the corresponding analytical solutions for the cases of $\theta_Y = 70^o, N = 32$ (left) and $\theta_Y = 110 ^ o, N = 32$(right). We can see that our algorithm give accurate results in both cases even on a relatively coarse mesh. More results on numerical errors are shown in Tables~\ref{110a}-\ref{70a}. We decrease the spacial mesh size and time step size proportionally. We can see that the error is of order $O(\delta t)$ when $\theta_Y=110^o,100^o$ and $70^o$. This is optimal as in
the standard diffusion generated method for smooth curves\cite{esedog2010diffusion}.
When $\theta_Y=90^o$, the convergence order is much better. This might be due to the error for time discretization is less dominant in this case, so that the error is of order $O(\delta x^2)$.
In all, the diffusion {generated} motion method gives optimal convergence rate $O(\delta t)$ with respect to the time step for
the wetting problems. This is  better than the
 threshold dynamics method developed in \cite{xuWangWang2016,xuWangWang2018}, where
the  convergence order is of $O(\delta t^{1/2})$, due to the existence of the contact points.

\begin{table}[!h]
\centering
\caption{Numerical errors  of Algorithm 1 for the case $\theta_Y= 110$ \degree}
\label{110a}
\begin{tabular}{cclc}
    \hline
      \hline
 \# Resolution & \# $\delta t$($t_{unit}$) & \# Error & \# Order \\
   \hline
  8$\times$8 & 1/2 & 9.96E-3 & - \\
 16$\times$16 & 1/4 & 5.31E-3 & 0.91 \\
  32$\times$32 & 1/8 & 2.68E-3 & 0.99 \\
   64$\times$64 & 1/16 & 1.67E-3 & 0.68 \\
    128$\times$128 & 1/32 & 7.42E-4 & 1.17 \\
  \hline
\end{tabular}
\end{table}
\begin{table}[!h]
\centering
\caption{Numerical errors  of Algorithm 1 for the case $\theta_Y= 100$ \degree}
\label{100a}
\begin{tabular}{cclc}
    \hline
      \hline
 \# Resolution & \# $\delta t$($t_{unit}$) & \# Error & \# Order \\
   \hline
  8$\times$8 & 1/2 & 1.42E-2 & - \\
 16$\times$16 & 1/4 & 6.03E-3 & 1.24 \\
  32$\times$32 & 1/8 & 2.26E-3 & 1.42 \\
   64$\times$64 & 1/16 & 1.04E-3 & 1.12 \\
    128$\times$128 & 1/32 & 4.58E-4 & 1.18 \\
  \hline
\end{tabular}
\end{table}

\begin{table}[!h]
\centering
\caption{Numerical errors  of Algorithm 1 for the case $\theta_Y= 90$ \degree}
\label{90a}
\begin{tabular}{cclc}
    \hline
      \hline
 \# Resolution & \# $\delta t$($t_{unit}$) & \# Error & \# Order \\
   \hline
  8$\times$8 & 1/2 & 2.39E-2 & - \\
 16$\times$16 & 1/4 & 6.17E-3 & 1.95\\
  32$\times$32 & 1/8 & 1.65E-3 & 1.90 \\
   64$\times$64 & 1/16 & 3.59E-4 & 2.20 \\
    128$\times$128 & 1/32 & {\color{black}6.59E-5} & 2.44 \\
  \hline
\end{tabular}
\end{table}

\begin{table}[!h]
\centering
\caption{Numerical errors  of Algorithm 1 for the case $\theta_Y= 70$ \degree}
\label{70a}
\begin{tabular}{cclc}
    \hline
      \hline
 \# Resolution & \# $\delta t$($t_{unit}$) & \# Error & \# Order \\
   \hline
  8$\times$8 & 1 & 2.86E-2 & - \\
 16$\times$16 & 1/2 & 1.09E-2 & 1.39\\
  32$\times$32 & 1/4 & 4.13E-3 & 1.40 \\
   64$\times$64 & 1/8 & 1.12E-3 & 1.99 \\
    128$\times$128 & 1/16 & 5.92E-4 & 0.92 \\
  \hline
\end{tabular}
\end{table}


\subsection{Wetting on  chemically patterned surfaces}
We  then compute a wetting problem on a chemically pattern solid surface by our algorithm. We assume  the
lower boundary $\Gamma$ is composed of two materials with different Young's angles.
Two typical material distributions are  shown in Figure \ref{k2esC} and \ref{k4esC}.  We use ${\mathcal{A}} $ to represent the red parts and $ {\mathcal{B}}$ to represent the green parts. Their corresponding Young's angles are respectively $\theta_{\mathcal{A}}=2\pi/3$ and $ \theta_{\mathcal{B}}=\pi/3$.

On the chemically patterned surface,
We first give an initial droplet and perform our algorithm until a stable state is reached (i.e. the energy in~\eqref{e:sharpPb1} is minimized).
After that, we increase(or decrease) the volume of the droplet a little by changing the level-set of
the signed distance function {slightly}. The new droplet may not be stable.
We use our algorithm again to compute a stable state.
 By repeating this process, we can observe the advancing(or receding) contact angle and
 the corresponding contact points. The difference between the advancing and receding trajectories gives the
  interesting contact angle hysteresis phenomena. More details will be given below.
%
%

We first consider the  $k = 2$ case, which implies  there are two green patterns in each side of
the middle point on the lower surface as shown in Figure~\ref{k2esC}.
The computational domain $\Omega$ is $(0,1)\times(0,1)$.
We solve this problem on a uniform  $ 256 \times 256 $ spacial mesh and the time step is fixed at $ \delta t = t_{unit}/40$.
In the advancing case, the initial droplet has a circular surface with two contact points $(0.45,0)$ and $(0.55,0) $ which are
located in the central red part. The initial contact angles are $ \theta_1=\theta_2 = \theta_\mathcal{A} = \frac {2} {3} \pi $.
Since the local contact angles equal to {Young's}  angle, the initial state of the droplet is stationary.
We add some volume to the initial state and solve the problem~\eqref{e:sharpPb1} by using {Algorithm 1} until we find
a solution. We repeat the process again and again until the distance $d$ of the contact points from  the middle point
$(0.5,0)$ is larger than $0.35$. In the process, the volume we added in each iteration is determined by the current position of the contact points. To be precise, let $2d$ be the distance between the two contact points, we add a volume of ${\pi d^2}/{20}$. In the receding case, the initial droplet is circular and has two contact points $(0.15,0)$ and $(0.85,0)$ which
are located in the outer red parts on the chemically patterned boundary. We set the initial contact angles as $ \theta_1=\theta_2 = \theta_\mathcal{A} = \frac {2} {3} \pi $. In this case, we decrease the volume gradually.
In each iteration, the volume we {decrease} is computed similarly as in the advancing case. We repeat this process until $ d< 0.05 $.

Figure \ref{k2esC} shows the position of the contact points(its distance $d$ away from the middle point) and the contact angle $\theta$ with respect to the volume of the droplet.  It is clear that  contact angle hysteresis phenomena
occur during the processes.  Both the trajectories of the position of the contact point and the contact angles
are different for the advancing(increasing volume) and receding(decreasing volume) cases.
The hysteresis occurs near some joint points of  the two materials. The largest advancing contact angle is equal to $\theta_\mathcal{A}$
and the smallest receding contact angle is equal to $\theta_\mathcal{B}$.
The results are consistent with the analytical analysis in \cite{xuwang2011}.
Figure~\ref{k2esC} shows the drop profiles for different volumes in the advancing case.
We see clearly the stick-slip phenomena of the contact points. In the receding case, the stick-slip behaviour of the droplet  is similar to the advancing case.
\begin{figure}[!h]
  \centering
  \includegraphics[width=7cm]{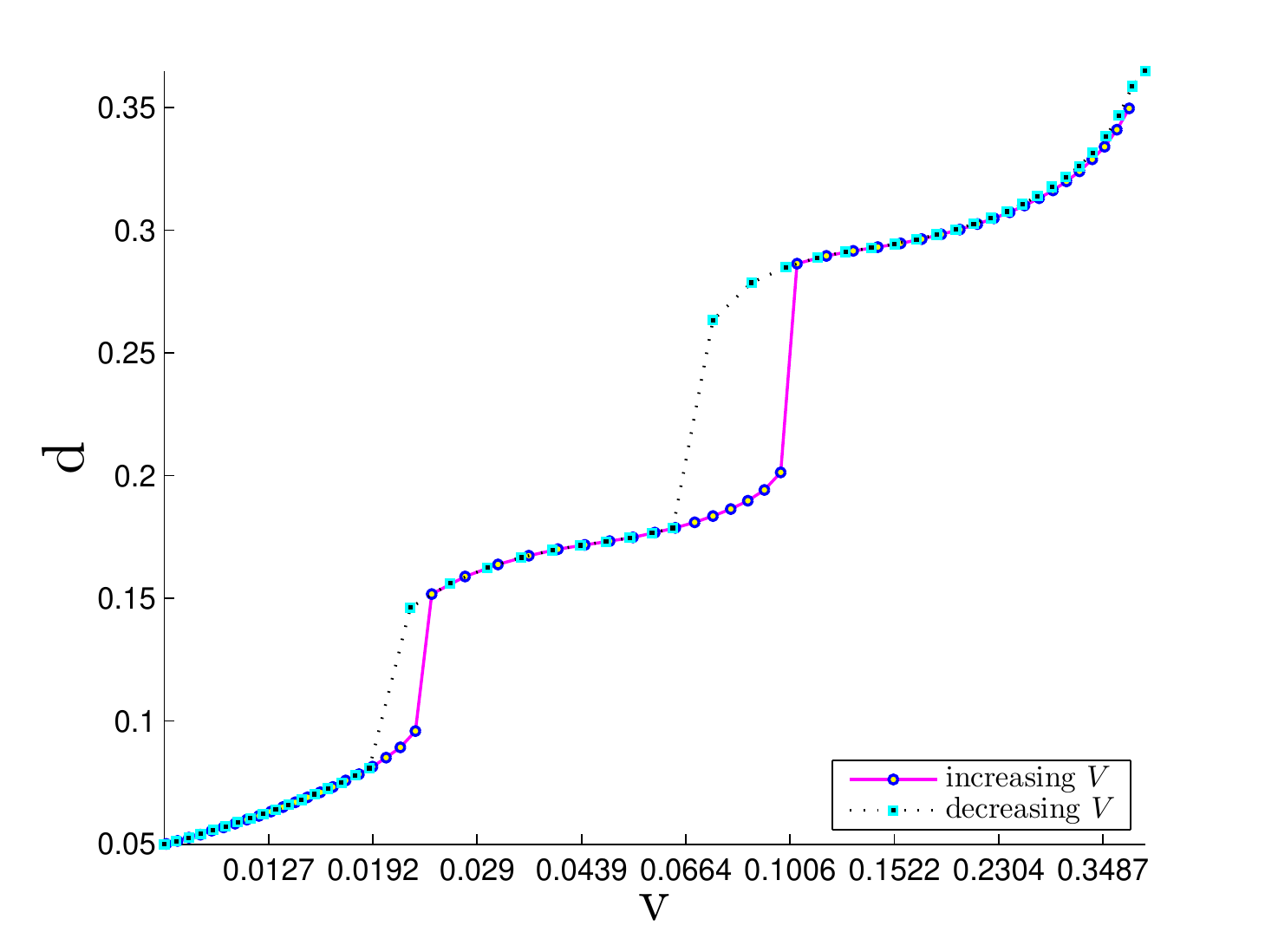}\qquad
  \includegraphics[width=7cm]{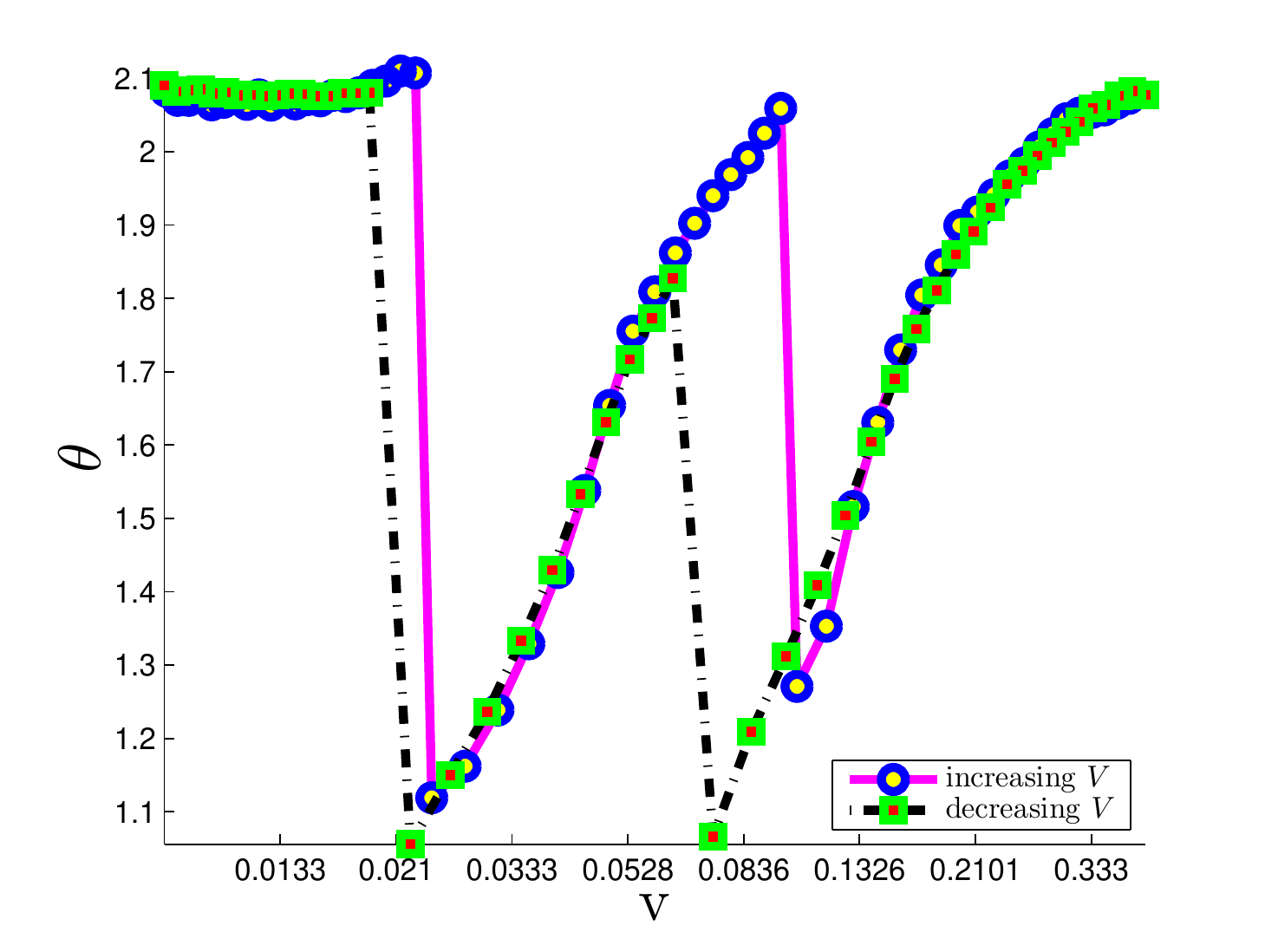}
  \caption{Contact angle hysteresis and constact line slipping($k=2$). Contact point (left) moved and contact angles (right) varied by the changing of drop volume.}\label{k2esmerge}
\end{figure}

\begin{figure}[!htp]
    \centering
    \includegraphics[height=7cm]{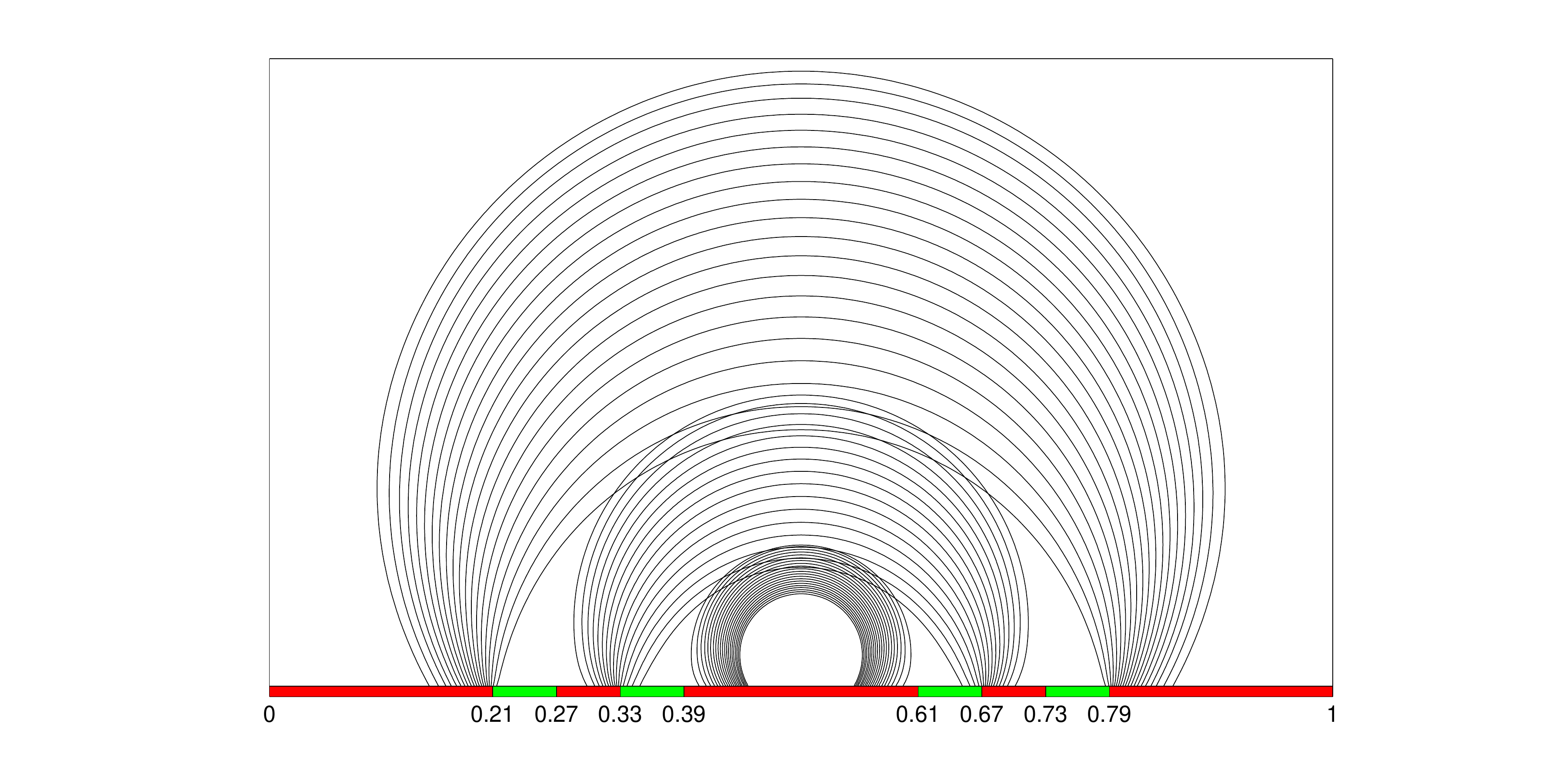}
\vspace{-0.5cm}
    \caption{The drop profiles with increasing volume when ${k=2}$.}
    \label{k2esC}
\end{figure}

In the  $ k = 4 $ case, the patterned surface is shown as in Figure~\ref{k4esC}.
Since there are more patterns on the lower boundary of $\Omega=(0,1)\times(0,1)$, we need finer meshes
to resolve the chemical patterns. In our experiments,
we adopt a $ 512 \times 512 $ spacial mesh and the time step is chosen as $ \delta t = t_{unit}/80$.
We do similar computations as in the $k=2$ case. The only difference is that the volume of the droplet
 changes more slowly. Suppose the current droplet has a circular surface with  two contact points, which are  of distance $d$ away from the middle point of the lower boundary. The volume we add (decrease) in the advancing(receding) case is $\pi d^2/40$.

In Figure \ref{k4esmerge} and  Figure \ref{k4esC}, we draw the pictures of $d$ and $\theta$ as {functions} of
the volume $v$ and the profiles of drops for the case $k=4$.
The results are quite similar to the case $k=2$. We observe the hysteresis in the trajectories of the
 contact angle and contact positions. There are also clear stick-slip phenomena of
 the contact points in the profile of the droplet. Since we have more patterns in this distribution of the materials on
 the substrate, we can see more stick-slip phenomena  on the joint points of the two materials.

\begin{figure}[!h]
  \centering
  \includegraphics[width=7cm]{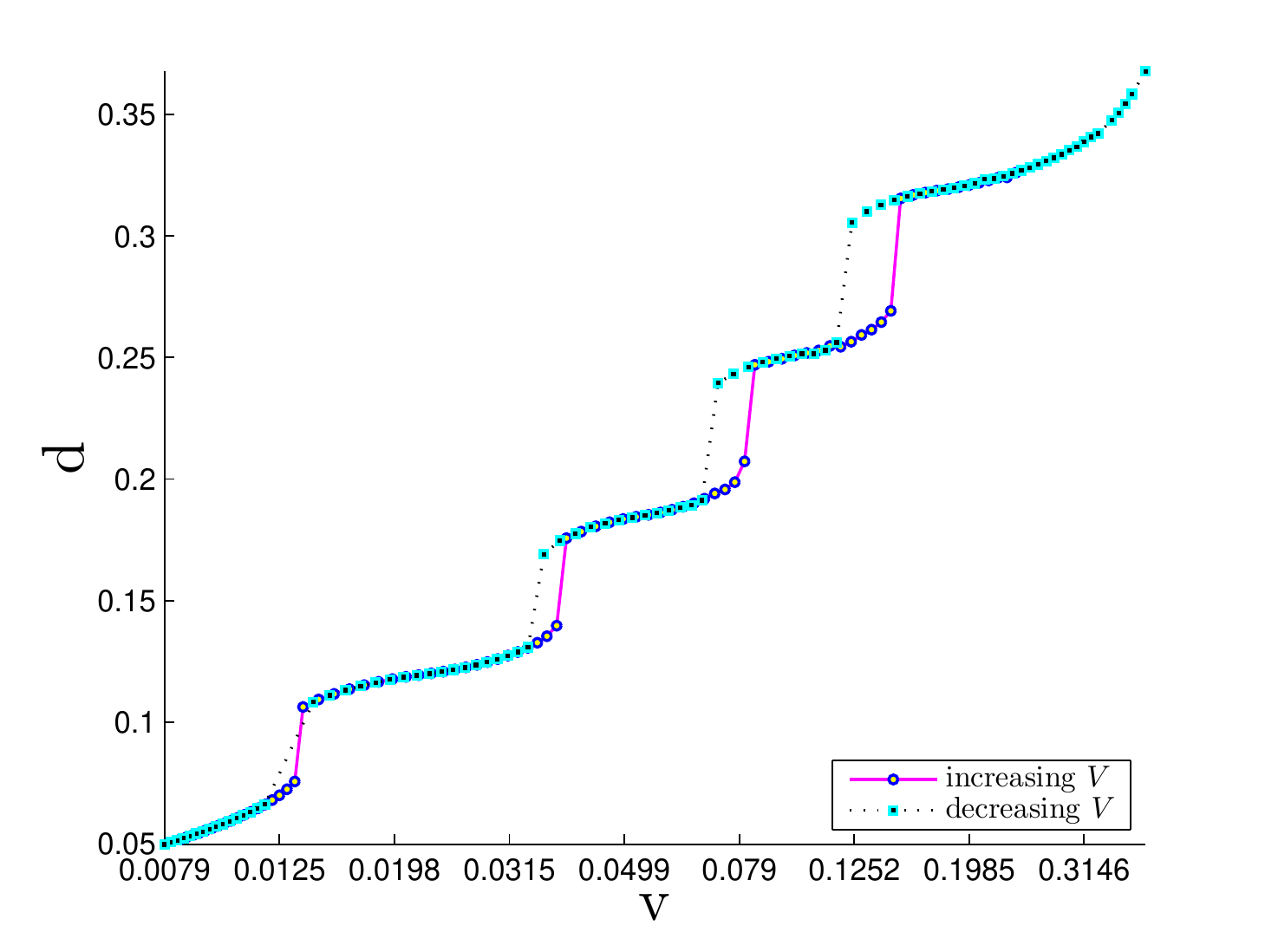}\qquad
  \includegraphics[width=7cm]{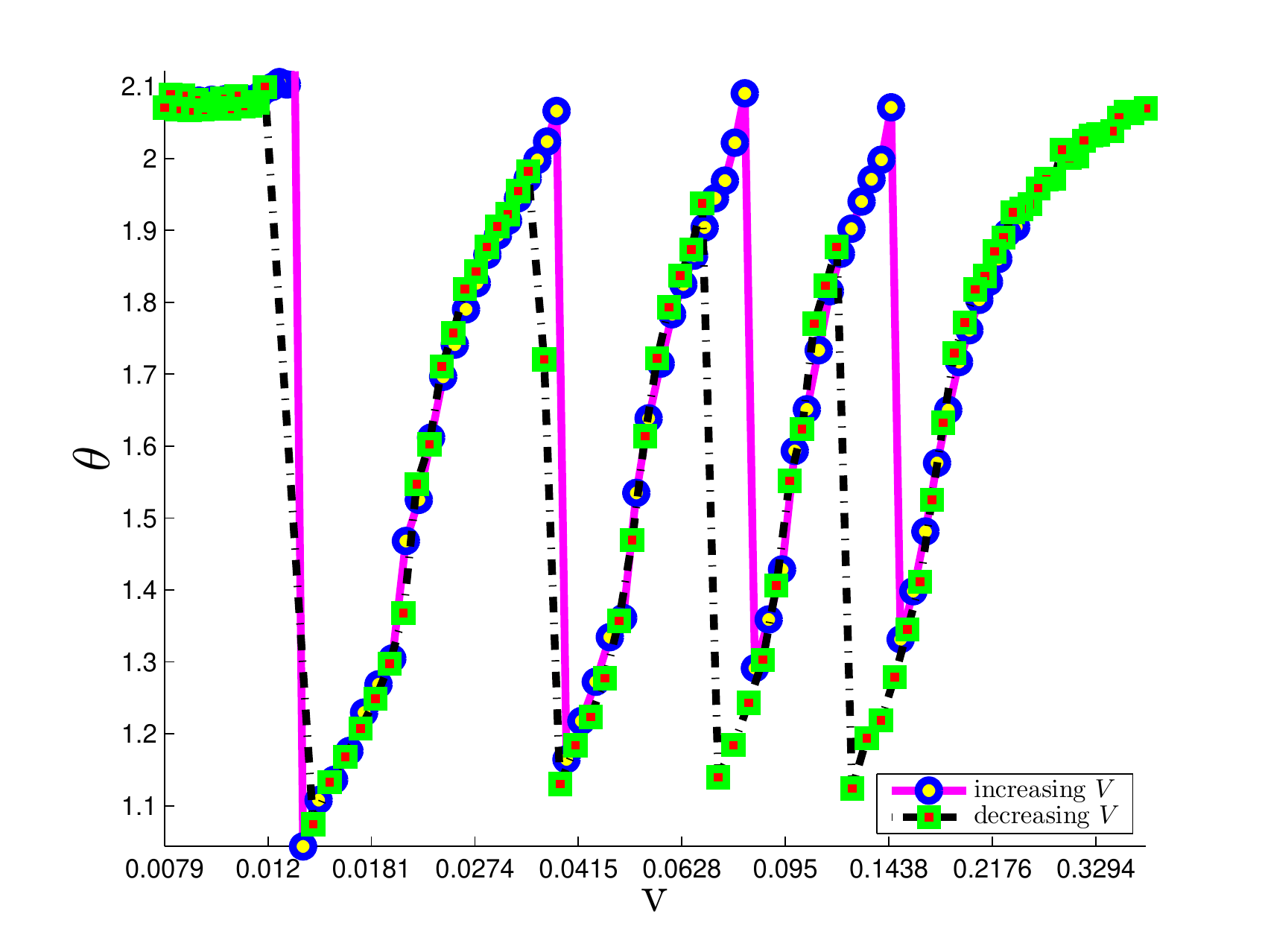}
  \caption{Contact angle hysteresis and constact line slipping($k=4$). Contact point (left) moved and contact angles (right) varied by the changing of drop volume.}\label{k4esmerge}
\end{figure}

\begin{figure}[!htp]
    \centering
    \includegraphics[height=7cm]{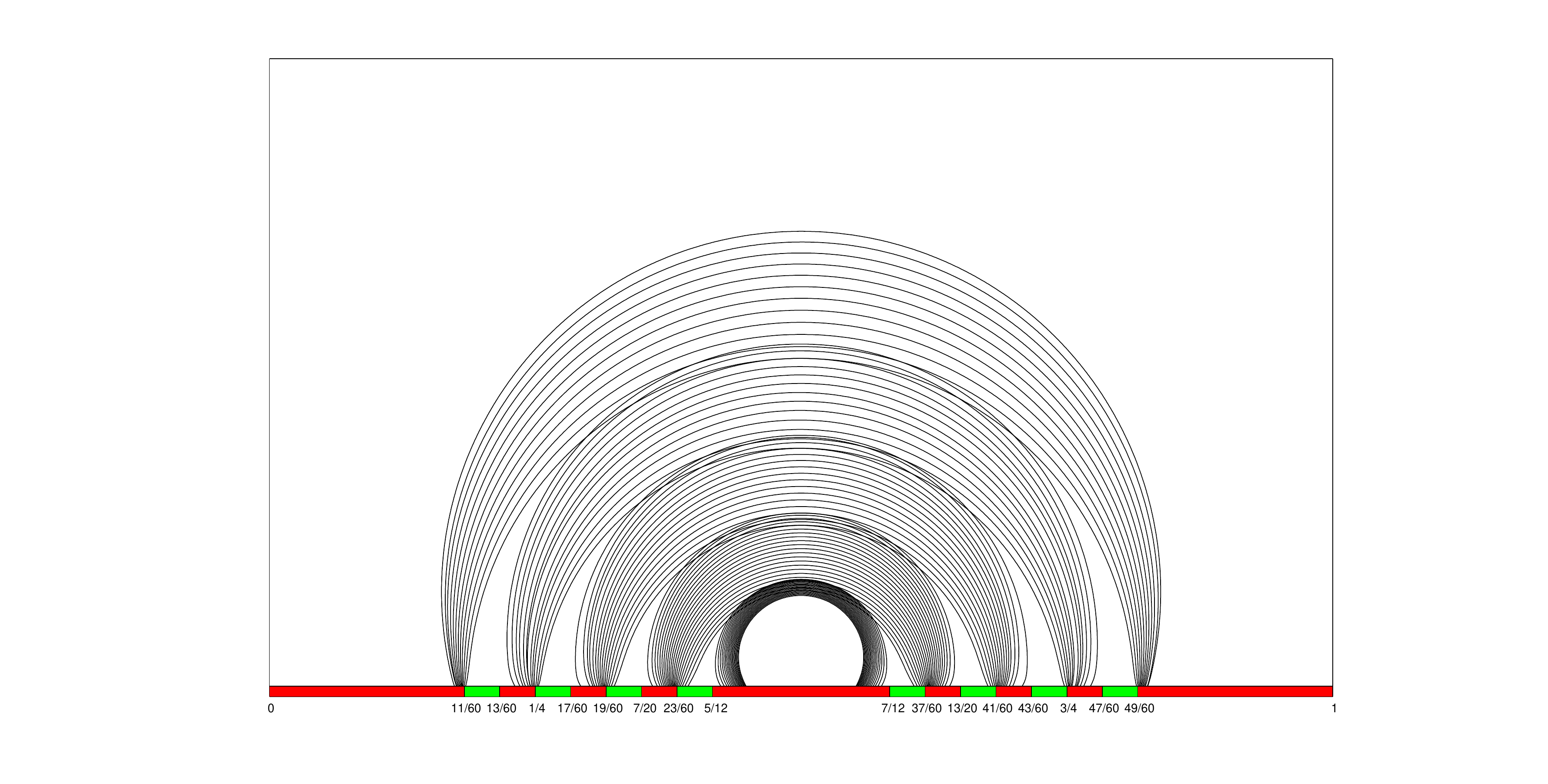}
\vspace{-0.5cm}
    \caption{The drop profiles with increasing volume in the $k=4$ case.}
    \label{k4esC}
\end{figure}

\section{Conclusions and Further discussions}
By using the Onsager principle as an approximation tool, we develop a new diffusion generated motion method for {the wetting problem}.  The key idea is to consider the leading order approximation of a  modified Allen-Cahn equation with a nonlinear relaxed boundary condition on {a solid surface}.
By assuming the leading order has a {\it tanh} profile, we derive a linear diffusion equation
for the signed distance function of the liquid-vapor interface.
The equation is much simpler than the original phase field model and can act as a basis
 to construct our numerical method.

In the proposed method,  we use the signed distance function to represent the interface between the liquid and vapor surface.
In each iteration, only a linear diffusion equation with a  linear boundary
condition is solved, in addition to  a re-distance step and a
volume correction step.
Numerical experiments show that the method has a convergence rate of $O(\delta t)$, which
is much better than the previous threshold dynamics method for wetting problems\cite{xuWangWang2016},
where one can only obtain a half order $O(\delta t^{1/2})$ accuracy.
The method can be seen as a generalization of the standard  diffusion generate motion
method using a signed distance function for mean curvature flows\cite{esedog2010diffusion}.
The energy stability of the method is analysed by careful studies for some geometric flows on
the substrates. Numerical results show that the method works well for wetting on inhomogeneous
surfaces.

In this paper, we focus on  {two-dimensional} problems. Our method can be  generalized to {three-dimensional} cases directly.
For three dimensional problems, it is very helpful to use an adaptive spacial mesh
to decrease the computational {complexity} as in \cite{xuying2020}.

\section*{Acknowledgement}
The work was partially supported by
NSFC 11971469 and by the National Key R\&D Program of China under Grant 2018YFB0704304 and Grant 2018YFB0704300.

\section*{References}
\bibliographystyle{unsrt}
\bibliography{onsager}
\end{document}